\documentclass[12pt]{article}

\usepackage{amsfonts}
\usepackage{amsmath}
\usepackage{amssymb}
\usepackage{amsthm}
\usepackage{xcolor}
\usepackage{color}

\usepackage{floatrow}
\usepackage{makeidx}
\usepackage{graphicx}
\usepackage{float}
\usepackage{mathtools}
\usepackage[margin=2cm]{geometry}
\usepackage{subcaption}

\usepackage{hyperref}  

\usepackage{tikz-cd}
\usetikzlibrary{graphs,decorations.pathmorphing,decorations.markings}

\newtheorem{theorem}{Theorem}[section]
\newtheorem{prop}[theorem]{Proposition}
\newtheorem{coro}[theorem]{Corollary}
\newtheorem{lemma}[theorem]{Lemma}

\newtheorem{example}{Example}[section]
\newtheorem{defin}{Definition}[section]
\newtheorem{remark}{Remark}[section]

\newcommand{\floor}[1]{\big\lfloor #1 \big\rfloor}

\newcommand{\pg}{\operatorname{PG}}

\title{Pyramid Transform of Manifold Data\\ via Subdivision Operators}

\author{Wael Mattar and Nir Sharon}
\date{}

\begin{document}
	\maketitle 
	
	\noindent\makebox[\linewidth]{\rule{\textwidth}{1pt}}
	
	\textbf{Abstract.} Multiscale transforms have become a key ingredient in many data processing tasks. With technological development, we observe a growing demand for methods to cope with non-linear data structures such as manifold values. In this paper, we propose a multiscale approach for analyzing manifold-valued data using a pyramid transform. The transform uses a unique class of downsampling operators that enable a non-interpolating subdivision schemes as upsampling operators. We describe this construction in detail and present its analytical properties, including stability and coefficient decay. Next, we numerically demonstrate the results and show the application of our method to denoising and anomaly detection.
	
	\noindent\makebox[\linewidth]{\rule{\textwidth}{1pt}}
	
	
	\section{Introduction}
	
	Many modern applications use manifold values as a primary tool to model data, e.g.,~\cite{frank2016continuous, lunga2014manifold, rahman2005multiscale}. Manifolds express a global nonlinear structure with constrained, high-dimensional elements. The employment of manifolds as data models raises the demand for computational methods to address fundamental tasks like integration, interpolation, and regression, which become challenging under the manifold setting, see, e.g.,~\cite{barbaresco2020lie, blanes2017concise, iserles2000lie, zeilmann2020geometric}. We focus on constructing a multiscale representation for manifold values using a fast pyramid transform. 
	
	Multiscale transforms are standard tools in signal and image processing that enable a hierarchical analysis of an object mathematically. Customarily, the first scale in the transform corresponds to a coarse representation, and as scales increase, so do the levels of approximation~\cite{mallat1999wavelet}. The pyramid transform uses a refinement or upsampling operator together with a corresponding subsampling operator for the construction of a fast multiscale representation of signals~\cite{donoho1992interpolating, wallner2020geometric}. The simplicity of this powerful method opened the door for many applications. Naturally, recent years found generalizations of multiscale representations for manifold values as well as manifold-valued pyramid transforms~\cite{storath2020wavelet}. Contrary to the classical, linear settings, where upsampling operators are often linear and global, e.g., polynomial interpolation, refinement operators to manifolds values are mostly nonlinear and local operators. One such class of operators arises in subdivision schemes.
	
	Subdivision schemes are powerful yet computationally efficient tools for producing smooth objects from discrete sets of points. These schemes are defined by repeatedly applying a subdivision operator that refines discrete sets. The subdivision refinements, which serve as upsampling operators, give rise to a natural connection between multiscale representations and subdivision schemes~\cite[Chapter 6]{daubechies1992ten}. In recent years, subdivision operators were adapted to manifold data and nonlinear geometries by various methods, and so have been their induced multiscale transforms, see~\cite{wallner2020geometric} for an overview.
	
	Multiscale transforms, based upon subdivision operators, commonly use interpolating subdivision schemes, i.e., operations that preserve the coarse objects through refinements. This standard facilitates the calculation of the missing detail coefficients at all scales. Particularly, coefficients associated with the interpolating values are systematically zeroed and do not have to be saved or processed in the following analysis levels. Therefore, this property of interpolating multiscale transforms makes the crux of many state-of-the-art algorithms, including data compression, see, e.g.,~\cite{lv2018novel}. On the other hand, the notion of non-interpolating pyramid transforms did not receive equal attention despite the popularly used non-interpolating subdivision schemes. A classic example of a widespread family of non-interpolating subdivision operators is the well-known B-spline, see~\cite{de1978practical, lane1980theoretical}.
	
	The main challenge behind constructing a non-interpolating pyramid transform revolves around the question of calculating the multiscale details. In particular, given a non-interpolating subdivision scheme, the corresponding subsampling operators involve applying infinitely-supported real-valued sequences. Therefore, care must be taken when realizing and implementing these operators. In this paper, we introduce a novel family of pyramid transforms suitable for non-interpolating subdivision schemes. Our multiscale transforms decompose manifold-valued sequences in a similar pyramidical fashion to the interpolating ones. Specifically, the non-interpolating transforms' construction relies on the recently-introduced decimation operators, see~\cite{dynlinear}, which are employed as subsampling operators. From the interpolating point of view, the decimation operation coincides with the simple downsampling operation, taking all even-indexed elements.
	
	Our contribution in this paper also covers the computability of the linear decimation operators and the linear operators' adaptation to cope with manifold data. In particular, it is not possible to implement decimation operators since they involve infinitely supported sequences. Therefore, we approximate these operators with affine averages of finite elements, which successfully lead to the desired mathematical results. We also derive an analytic condition for decimation operators termed ``decimation-safety,'' and show that all our operators over manifold data satisfy the condition. We prove that multiscale transforms associated with decimation-safe operators, together with their corresponding inverses, enjoy coefficient decay and stability properties. The outcomes are an essential feature of the multiscale transform of manifold data and can significantly contribute to various applications and scientific questions.
	
	We conclude the paper with several numerical demonstrations of both the theoretical results we obtained and applications of data processing. Specifically, we provide examples of the use of our method for denoising and anomaly detection on synthetically generated manifold data. All figures and examples were generated using a code package that complements the paper and is available online for reproducibility.
	
	
	\section{Preliminaries}
	
	\subsection{Linear univariate subdivision schemes}
	
	In the functional setting, a linear binary subdivision scheme $\mathcal{S}$ operates on a real-valued bi-infinite sequence $\boldsymbol{c}=\left\{ c_k\in\mathbb{R}\mid k\in\mathbb{Z}\right\}$. Applying the subdivision scheme on $\boldsymbol{c}$ yields a sequence $\mathcal{S}(\boldsymbol{c})$ which is associated with the values over the refined grid $2^{-1}\mathbb{Z}$. This process is repeated infinitely and results in values defined on the dyadic rationals, which form a dense set over the real line. If the generated values of the repeated process for any sequence $\boldsymbol{c}$ converge uniformly at the dyadic points to the values of a continuous function, we term the subdivision scheme convergent and treat the function as its limit, see, e.g.,~\cite{dyn2002analysis}. We denote a linear, binary refinement rule of a univariate subdivision scheme $\mathcal{S}$ with a finitely supported mask $\boldsymbol{\alpha}$ by
	\begin{align}~\label{LinearSubdivisionScheme}
		\mathcal{S}_{\boldsymbol{\alpha}}(\boldsymbol{c})_k = \sum_{i\in\mathbb{Z}}\alpha_{k-2i}c_i, \quad k\in\mathbb{Z}.
	\end{align}
	Depending on the parity of the index $k$, the refinement rule~\eqref{LinearSubdivisionScheme} can be split into two rules. Namely,
	\begin{align}~\label{LinearCentersOfMass}
		\mathcal{S}_{\boldsymbol{\alpha}}(\boldsymbol{c})_{2k}=\sum_{i\in\mathbb{Z}}\alpha_{2i}c_{k-i}\quad \text{and}\quad \mathcal{S}_{\boldsymbol{\alpha}}(\boldsymbol{c})_{2k+1}=\sum_{i\in\mathbb{Z}}\alpha_{2i+1}c_{k-i}, \quad k\in\mathbb{Z}.
	\end{align}
	For more details, we encourage the reader to see~\cite{dyn2002interpolatory}. Moreover, the scheme~\eqref{LinearSubdivisionScheme} can be written as the convolution $\mathcal{S}_{\boldsymbol{\alpha}}(\boldsymbol{c})=\boldsymbol{\alpha}\ast(\boldsymbol{c}\uparrow 2)$ where
	\begin{align*}
		\left(\boldsymbol{c}\uparrow 2\right)_k=\begin{cases}
			c_{k/2}, & k \text{ is even}, \\
			0, & \text{otherwise},
		\end{cases} \quad
		k\in\mathbb{Z},
	\end{align*}
	is the \emph{upsampled sequence} $\boldsymbol{c}$.
	
	A subdivision scheme is termed interpolating if $\mathcal{S}_{\boldsymbol{\alpha}}(\boldsymbol{c})_{2k}=c_{k}$ for all $k\in\mathbb{Z}$. Equivalently, $\mathcal{S}_{\boldsymbol{\alpha}}(\boldsymbol{c})=\boldsymbol{c}\uparrow 2$ over the even indices. A necessary condition for the convergence of a subdivision scheme with the refinement rule~\eqref{LinearSubdivisionScheme}, see e.g.,~\cite{dyn1992subdivision}, is
	\begin{align}~\label{AlphaInvariance}
		\sum_{i\in\mathbb{Z}}\alpha_{2i}=\sum_{i\in\mathbb{Z}}\alpha_{2i+1}=1.
	\end{align}
	Henceforth, we assume that any subdivision operator mentioned is of convergent subdivision schemes. Moreover, we refer to the masks which satisfy~\eqref{AlphaInvariance} as \emph{shift invariant}. The reason being is that applying a subdivision scheme with shift invariant mask on a shifted data points results with precisely the shifted original outcome. Note that an invariant rule~\eqref{LinearCentersOfMass} is a weighted average which can be interpreted as the \emph{center of mass} of elements $c_{k-i}$, with the components of $\boldsymbol{\alpha}$ as their weights. This interpretation is fundamental for the adaptation of linear subdivision schemes to manifold data, as we present next.
	
	\subsection{The Riemannian analogue of a linear subdivision scheme}\label{RiemannianAnalogue}
	
	Subdivision schemes with shift invariant masks were adapted to manifold-valued data via different methods and approaches, see e.g.,~\cite{dyn2017manifold, grohs2010general, rahman2005multiscale, wallner2005convergence}. One natural extension of a linear subdivision scheme~\eqref{LinearSubdivisionScheme} to manifold-valued data can be done with the help of the Riemannian structure. Let $\mathcal{M}$ be a Riemannian manifold equipped with Riemannian metric which we denote by $\langle\cdot,\cdot\rangle$. The Riemannian geodesic distance $\rho(\cdot,\cdot)\colon\mathcal{M}^2\to\mathbb{R}^+$ is
	\begin{align}~\label{RiemannianDistance}
		\rho(x,y)=\inf_{\Gamma}\int_{a}^{b}|\dot{\Gamma}(t)|dt,
	\end{align}
	where $\Gamma\colon[a,b]\to\mathcal{M}$ is a curve connecting points $\Gamma(a)=x$ and $\Gamma(b)=y$, and $|\cdot|^2 = \langle \cdot, \cdot \rangle$.
	
	The linear subdivision scheme~\eqref{LinearSubdivisionScheme} associated with an invariant mask~\eqref{AlphaInvariance} can be characterized as the unique solution of the optimization problem
	\begin{align}~\label{LinearFrechetMean}
		\mathcal{S}_{\boldsymbol{\alpha}}(\boldsymbol{c})_{k} = \text{arg}\min_{x\in\mathbb{R}}\sum_{i\in\mathbb{Z}}\alpha_{k-2i}\|x-c_{i}\|^2, \quad k\in\mathbb{Z},
	\end{align}
	where $\|\cdot\|$ denotes the standard Euclidean norm. This is an alternative formulation for~\eqref{LinearSubdivisionScheme} as the Euclidean center of mass.
	
	For an $\mathcal{M}$-valued sequence $\boldsymbol{c}$ we transfer the optimization problem~\eqref{LinearFrechetMean} to $\mathcal{M}$ by replacing the Euclidean distance with the Riemannian geodesic distance~\eqref{RiemannianDistance}. We denote by $\mathcal{T}_{\boldsymbol{\alpha}}$ the Riemannian analogue of the linear subdivision scheme $\mathcal{S}_{\boldsymbol{\alpha}}$; given a mask $\boldsymbol{\alpha}$, we define the adapted subdivision scheme as
	\begin{align}~\label{ManifoldFrechetMean}
		\mathcal{T}_{\boldsymbol{\alpha}}(\boldsymbol{c})_{k} = \text{arg}\min_{x\in\mathcal{M}}\sum_{i\in\mathbb{Z}}\alpha_{k-2i}\rho(x,c_i)^2, \quad k\in\mathbb{Z}.
	\end{align}
	When the solution of~\eqref{ManifoldFrechetMean} exists uniquely, we term the solution as the \emph{Riemannian center of mass}~\cite{grove1973conjugatec}. It is also termed Karcher mean for matrices and Fr\`{e}chet mean in more general metric spaces, see~\cite{Karcher2014riemannian}.
	
	The global well-definedness of~\eqref{ManifoldFrechetMean} when $\alpha_k\geq 0$ is studied in~\cite{kobayashi1963foundations}. Moreover, in the framework where $\mathcal{M}$ has a non-positive sectional curvature, if the mask $\boldsymbol{\alpha}$ is shift invariant, then a globally unique solution for problem~\eqref{ManifoldFrechetMean} can be found, see e.g.,~\cite{hardering2015intrinsic,karcher1977riemannian,sander2016geodesic}. Recent studies of manifolds with positive sectional curvature show necessary conditions for uniqueness on the spread of points with respect to the injectivity radius of $\mathcal{M}$~\cite{dyer2016barycentric, SvenjaWallnerConvergenceSphere}. We focus our attention on $\mathcal{M}$-valued sequences $\boldsymbol{c}$ that are admissible in the sense that $\mathcal{T}_{\boldsymbol{\alpha}}(\boldsymbol{c})$ is uniquely defined for any shift invariant mask $\boldsymbol{\alpha}$, i.e., problems~\eqref{ManifoldFrechetMean} have unique solutions.
	
	We interpret many alternative methods for adapting subdivision operators to manifolds as finite approximations for the Riemmanian center of mass~\eqref{ManifoldFrechetMean}. This includes, for example, the exp-log methods~\cite{grohs2012definability, rahman2005multiscale}, repeated binary averaging~\cite{dyn2017global, wallner2005convergence}, and inductive means~\cite{dyn2017manifold}.
	
	\subsection{Interpolating linear multiscale transform}
	
	The notion of pyramid transforms is to represent a high-resolution sequence of data points as a pyramid consisting of a coarse approximation in addition to the multiscale layers, each corresponding to a different scale, see, e.g.,~\cite{grohs2010stability, dynlinear}. In this section, we briefly review an interpolating multiscale transform, see, e.g.,~\cite{donoho1992interpolating,harten1996multiresolution}.
	
	With the help of an interpolating subdivision scheme $\mathcal{S}_{\boldsymbol{\alpha}}$, a high resolution real-valued sequence $\boldsymbol{c}^{(1)}$ associated with the values over the grid $2^{-1}\mathbb{Z}$ can be decomposed into a coarse (low resolution) sequence $\boldsymbol{c}^{(0)}$ over the integers together with a sequence of detail coefficients $\boldsymbol{d}^{(1)}$ over the grid $2^{-1}\mathbb{Z}$ by letting 
	\begin{align}~\label{LinearInterpolatoryDecomposition}
		\boldsymbol{c}^{(0)}=\boldsymbol{c}^{(1)}\downarrow 2 \quad \text{and} \quad \boldsymbol{d}^{(1)}=\boldsymbol{c}^{(1)}-\mathcal{S}_{\boldsymbol{\alpha}} \boldsymbol{c}^{(0)},
	\end{align}
	where $\downarrow 2$ is the \emph{downsampling} operator given by $\left(\boldsymbol{c}\downarrow 2\right)_{k}=c_{2k}$ for all $k\in\mathbb{Z}$. In the same manner of decomposition~\eqref{LinearInterpolatoryDecomposition}, given a real-valued sequence $\boldsymbol{c}^{(J)}$, $J\in\mathbb{N}$ associated with the values over the fine grid $2^{-J}\mathbb{Z}$, it can be recursively decomposed by
	\begin{align}~\label{LinearInterpolatoryMultiscale}
		\boldsymbol{c}^{(\ell-1)}=\boldsymbol{c}^{(\ell)}\downarrow 2, \quad \boldsymbol{d}^{(\ell)}=\boldsymbol{c}^{(\ell)}-\mathcal{S}_{\boldsymbol{\alpha}} \boldsymbol{c}^{(\ell-1)}, \quad \ell=1,2,\dots,J.
	\end{align}
	
	The process~\eqref{LinearInterpolatoryMultiscale} yields a pyramid of sequences $\left\{\boldsymbol{c}^{(0)};\boldsymbol{d}^{(1)},\dots,\boldsymbol{d}^{(J)}\right\}$ where $\boldsymbol{c}^{(0)}$ is the coarse approximation coefficients given over the integers, and $\boldsymbol{d}^{(\ell)}$, $\ell=1,2,\dots,J$ are the detail coefficients at level $\ell$, given over the values of the grids $2^{-\ell}\mathbb{Z}$. We obtain synthesis by the following iterations,
	\begin{align}~\label{LinearInterpolatoryReconstruction}
		\boldsymbol{c}^{(\ell)}=\mathcal{S}_{\boldsymbol{\alpha}} \boldsymbol{c}^{(\ell-1)}+\boldsymbol{d}^{(\ell)}, \quad \ell=1,2,\dots,J,
	\end{align}
	which is the inverse transform of~\eqref{LinearInterpolatoryMultiscale}. At index $k\in\mathbb{Z}$, the detail coefficient $d^{(\ell)}_k$ measures the agreement between $c^{(\ell)}_k$ and $(\mathcal{S}_{\boldsymbol{\alpha}} \boldsymbol{c}^{(\ell-1)})_k$. In particular, since $\mathcal{S}_{\boldsymbol{\alpha}}$ is interpolating, we have that $d^{(\ell)}_{2k}=0$ for all $k\in\mathbb{Z}$, that is, 
	\begin{align}~\label{InterpolatingCondition}
		\big[(\mathcal{I}-\mathcal{S}_{\boldsymbol{\alpha}} \downarrow 2)\boldsymbol{c}^{(\ell)}\big]\downarrow 2 = \boldsymbol{0}, \quad \ell=1,2,\dots,J,
	\end{align}
	where $\mathcal{I}$ is the identity operator in the functional setting. Therefore, property~\eqref{InterpolatingCondition} allows us to omit ``half'' of the detail coefficients of each layer as we represent real-valued sequences -- a natural benefit for data compression. The diagrams of Figure~\ref{LinearInterpolatingDiagram} demonstrate the interpolating multiscale transforms like~\eqref{LinearInterpolatoryMultiscale} and its inverse.
	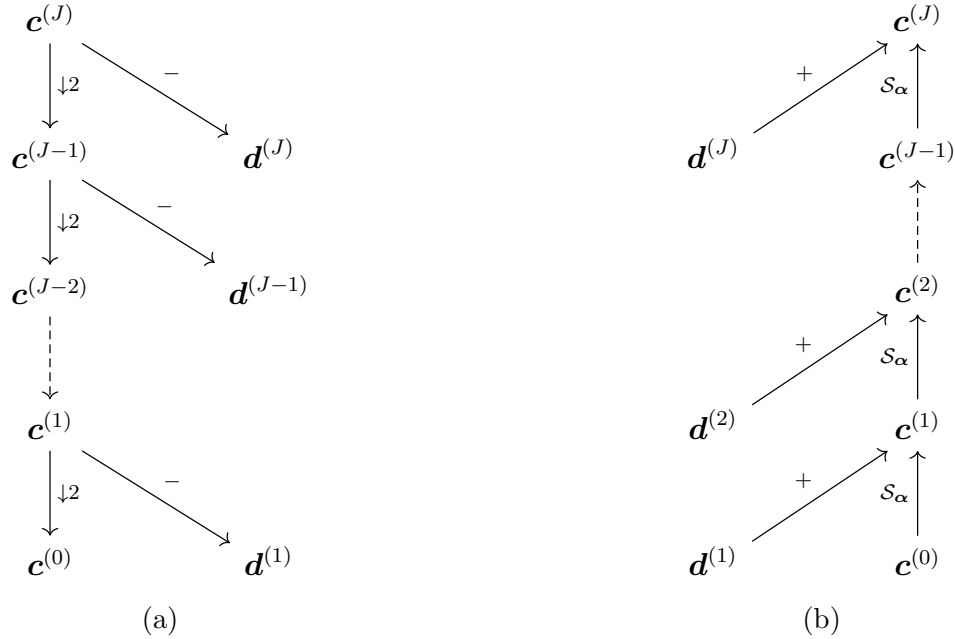
\begin{figure}[H]
		\centering
		\begin{subfigure}{.5\textwidth}
			\centering
			\begin{tikzcd}[sep=large]
				\boldsymbol{c}^{(J)} \arrow[d, "\downarrow 2"] \arrow[dr, "-"]
				\\
				\boldsymbol{c}^{(J-1)} \arrow[d, "\downarrow 2"] \arrow[dr, "-"] & \boldsymbol{d}^{(J)} \\
				\boldsymbol{c}^{(J-2)} \arrow[d, dashed] & \boldsymbol{d}^{(J-1)} \\
				\boldsymbol{c}^{(1)} \arrow[d, "\downarrow 2"] \arrow[dr,"-"]
				\\
				\boldsymbol{c}^{(0)} & \boldsymbol{d}^{(1)}
			\end{tikzcd}
			\caption{}
			\label{Analysis}
		\end{subfigure}%
		\begin{subfigure}{.5\textwidth}
			\centering
			\begin{tikzcd}[sep=large]
				& \boldsymbol{c}^{(J)} \\
				\boldsymbol{d}^{(J)} \arrow[ur, "+"] & \boldsymbol{c}^{(J-1)} \arrow[u, "\mathcal{S}_{\boldsymbol{\alpha}}"] \\
				& \boldsymbol{c}^{(2)} \arrow[u, dashed] \\
				\boldsymbol{d}^{(2)} \arrow[ur, "+"] & \boldsymbol{c}^{(1)} \arrow[u, "\mathcal{S}_{\boldsymbol{\alpha}}"] \\
				\boldsymbol{d}^{(1)} \arrow[ur, "+"] & \boldsymbol{c}^{(0)} \arrow[u, "\mathcal{S}_{\boldsymbol{\alpha}}"]
			\end{tikzcd}
			\caption{}
			\label{Synthesis}
		\end{subfigure}
		\caption{The pyramid transform. On the left, the analysis~\eqref{LinearInterpolatoryMultiscale}, on the right, the synthesis~~\eqref{LinearInterpolatoryReconstruction}.}
		\label{LinearInterpolatingDiagram}
	\end{figure}
	
	In fact, the interpolating multiscale transform~\eqref{LinearInterpolatoryMultiscale} is a special case of the family of transforms presented in~\cite{donoho1992interpolating}. In particular, the operators $\mathcal{S}_{\boldsymbol{\alpha}}$ and $\downarrow 2$ play the roles of \emph{upscaling} and \emph{downscaling} filters, respectively.
	
	\subsection{Non-interpolating linear multiscale transform}
	
	The difficulty in using non-interpolating upscaling operators $\mathcal{S}_{\boldsymbol{\alpha}}$ in multiscale like~\eqref{LinearInterpolatoryMultiscale}, is that the sequence $\mathcal{S}_{\boldsymbol{\alpha}}(\boldsymbol{c})$ does not preserve the elements $\boldsymbol{c}$. In such case, the details must include more than just the difference between the original sequence $\boldsymbol{c}$ and refined downsampled sequence $\mathcal{S}_{\boldsymbol{\alpha}}(\boldsymbol{c}\downarrow 2)$. 
	
	The extension of multiscale transforms from interpolating subdivision operators to a wider class of subdivision operators involves \emph{even-reversible} operators. Each of these operators helps recover, after one iteration of refinement, data points associated with even indices. In other words, given a subdivision operator $\mathcal{S}_{\boldsymbol{\alpha}}$, we seek for an operator $\mathcal{D}$ such that
	\begin{align}~\label{NonInterpolatingCondition}
		\big[(\mathcal{I}-\mathcal{S}_{\boldsymbol{\alpha}}\mathcal{D})\boldsymbol{c}\big]\downarrow 2 = \boldsymbol{0}
	\end{align}
	holds for any real-valued sequence $\boldsymbol{c}$. Indeed, condition~\eqref{NonInterpolatingCondition} is the analogue of condition~\eqref{InterpolatingCondition}. Specifically, $\downarrow 2$ is replaced with the operator $\mathcal{D}$.
	
	Let $\boldsymbol{\gamma}$ be a sequence such that $\sum_{i\in\mathbb{Z}}\gamma_i=1$. Then, for any real-valued sequence $\boldsymbol{c}$ we define the \emph{decimation operator} $\mathcal{D}_{\boldsymbol{\gamma}}$ associated with the sequence $\boldsymbol{\gamma}$ to be
	\begin{align}~\label{LinearDecimationScheme}
		\mathcal{D}_{\boldsymbol{\gamma}} (\boldsymbol{c})_k=\sum_{i\in\mathbb{Z}}\gamma_{k-i}c_{2i}, \quad k\in\mathbb{Z}.
	\end{align}
	
	Indeed, decimation operators are downscaling operators in the sense that applying them to a sequence of data results in fewer data. Note that $\boldsymbol{\gamma}$ can have infinite support. Thus, calculating~\eqref{LinearDecimationScheme} usually involves truncation errors. The rule~\eqref{LinearDecimationScheme} can be expressed as the unique solution of an optimization problem as similar to~\eqref{LinearFrechetMean}, where $\alpha_{k-2i}$ and $c_i$ are replaced with $\gamma_{k-i}$ and $c_{2i}$, respectively. Moreover, it can also be expressed in terms of the convolutional equation, 
	\[\mathcal{D}_{\boldsymbol{\gamma}} (\boldsymbol{c}) = \boldsymbol{\gamma} \ast (\boldsymbol{c}\downarrow 2).\] 
	Note that in the interpolation case, $\mathcal{D}_{\boldsymbol{\delta}}$ agrees with $\downarrow 2$ where $\boldsymbol{\delta}$ is the Kronecker delta sequence, $\delta_0=1$ and $ \delta_i=0$ for $i\neq 0$.
	
	The unique solution of~\eqref{NonInterpolatingCondition}, see~\cite{dynlinear}, is the decimation operator $\mathcal{D}_{\boldsymbol{\gamma}}$ where $\boldsymbol{\gamma}$ is found via the convolutional equation
	\begin{align}~\label{ConvolutionalEquation}
		\boldsymbol{\gamma} \ast (\boldsymbol{\alpha}\downarrow 2) = \boldsymbol{\delta}.
	\end{align}
	Using Wiener's Lemma~\cite{grochenig2010wiener}, if $\boldsymbol{\alpha}\downarrow 2$ is compactly supported, then such $\boldsymbol{\gamma}$ with infinite support exists. In this case, we say that $\mathcal{D}_{\boldsymbol{\gamma}}$ is the even-inverse of $\mathcal{S}_{\boldsymbol{\alpha}}$. Furthermore, $\boldsymbol{\gamma}$ decays geometrically, as shown in~\cite{strohmer2002four}. More precisely,
	\begin{align}~\label{GammaCoefficientsDecay}
		|\gamma_k|\leq C\lambda^{|k|}, \quad k\in\mathbb{Z},
	\end{align}
	for constants $C>0$ and $0<\lambda<1$. This bound on the decay rate is essential for the computation of the decimation operation $\mathcal{D}_{\boldsymbol{\gamma}}$, as we will see in the next section. We proceed with two examples of subdivision schemes which generate B-spline curves. First, we invoke the general formula of their compact masks. The mask $\boldsymbol{\alpha}^{[m]}$ of the B-spline subdivision operator of order $m\in\mathbb{N}$ is given by
	\[ \boldsymbol{\alpha}^{[m]}_{k- \lceil m/2 \rceil} = 2^{-m}{\binom{m+1}{k}}, \quad k=0,1,\dots,m+1. \]
	For more details see~\cite{dyn1992subdivision}.
	
	\begin{example}~\label{QuadraticExample}
		\emph{(The quadratic B-spline).} Consider the mask $\boldsymbol{\alpha}^{[2]}=[\alpha_{-1},\alpha_{0},\alpha_{1},\alpha_{2} ]=\frac{1}{4}[1,3,3,1]$, then the downsampled mask is $\boldsymbol{\alpha}\downarrow 2=\frac{1}{4}[3, 1]$ and the solution of the corresponding convolutional equation~\eqref{ConvolutionalEquation} is
		\begin{align*}
			\gamma_k =\begin{cases}
				\displaystyle\frac{4}{3}\bigg(-\frac{1}{3}\bigg)^k, & k=0,1,2,\dots, \\
				0, & \text{otherwise.}
			\end{cases}
		\end{align*}
		This subdivision scheme is also known as the \emph{corner-cutting} scheme.
	\end{example}
	\begin{example}~\label{CubicExample}
		\emph{(The cubic B-spline).} The next scheme generates cubic B-splines and its mask is given as $\boldsymbol{\alpha}^{[3]}=[\alpha_{-2},\alpha_{-1},\alpha_{0},\alpha_{1},\alpha_{2} ]=\frac{1}{8}[1,4,6,4,1]$, then the downsampled mask is $\boldsymbol{\alpha}\downarrow 2=\frac{1}{8}[1, 6, 1]$ and the solution of the corresponding convolutional equation~\eqref{ConvolutionalEquation} is
		\begin{align*}
			\big[\dots,\gamma_{-3},\gamma_{-2},\gamma_{-1},\gamma_{0},\gamma_{1},\gamma_{2},&\gamma_{3}\dots\big]=\\&
			\big[\dots,-0.0071,0.0416,-0.2426,1.4142,-0.2426,0.0416,-0.0071\dots\big].
		\end{align*}
	\end{example}
	Note that in Examples~\ref{QuadraticExample} and~\ref{CubicExample}, $\boldsymbol{\gamma}$ is infinitely and bi-infinitely supported, respectively.
	
	We are finally in a position to present the non-interpolating linear multiscale transform. Given a non-interpolating subdivision scheme $\mathcal{S}_{\boldsymbol{\alpha}}$ with its corresponding even-inverse decimation operator $\mathcal{D}_{\boldsymbol{\gamma}}$, and a sequence $\boldsymbol{c}^{(J)}$ associated with the values over the fine grid $2^{-J}\mathbb{Z}$, we consider the following non-interpolating multiscale transform,
	\begin{align}~\label{LinearMultiscaleTransform}
		\boldsymbol{c}^{(\ell-1)}=\mathcal{D}_{\boldsymbol{\gamma}} \boldsymbol{c}^{(\ell)}, \quad \boldsymbol{d}^{(\ell)}=\boldsymbol{c}^{(\ell)}-\mathcal{S}_{\boldsymbol{\alpha}} \boldsymbol{c}^{(\ell-1)}, \quad \ell = 1,2,\dots,J.
	\end{align}
	Iterating~\eqref{LinearMultiscaleTransform} yields a pyramid of data $\left\{\boldsymbol{c}^{(0)};\boldsymbol{d}^{(1)},\dots,\boldsymbol{d}^{(J)}\right\}$ as similar to the interpolating transform~\eqref{LinearInterpolatoryMultiscale}. We obtain synthesis again by~\eqref{LinearInterpolatoryReconstruction}.
	
	It turns out that the multiscale transform~\eqref{LinearMultiscaleTransform} enjoys two main properties, that are, decay of the detail coefficients and stability of the inverse transform. Here we invoke both results citing~\cite{dynlinear}, but first, we define the operator $\Delta$ on a sequence $\boldsymbol{c}$ to be $\Delta\boldsymbol{c} = \sup_{k \in \mathbb{Z}} |c_{k+1}-c_k|$, and call decimation sequences $\boldsymbol{\gamma}$ that sum to $1$ as shift invariant. Furthermore, we recall that the $\infty$-norm of a real-valued sequence $\boldsymbol{v}=\left\{ v_k\mid k\in\mathbb{Z}\right\}$ is defined by $\|\boldsymbol{v}\|_\infty = \sup_{k \in \mathbb{Z}}|v_k|$.
	\begin{theorem}~\label{LinearGammaDecayTHM}
		Let $\boldsymbol{c}^{(J)}$ be a real-valued sequence and denote by $\left\{\boldsymbol{c}^{(0)};\boldsymbol{d}^{(1)},\dots,\boldsymbol{d}^{(J)}\right\}$ its multiscale transform generated by~\eqref{LinearMultiscaleTransform}. The linear subdivision scheme is $\mathcal{S}_{\boldsymbol{\alpha}}$ and $\mathcal{D}_{\boldsymbol{\gamma}}$ is its corresponding decimation operator defined using a shift invariant sequence $\boldsymbol{\gamma}$. We further assume that $ K_{\boldsymbol{\gamma}}=2\sum_{i\in\mathbb{Z}}|\gamma_i||i|<\infty$.
		Then,
		\begin{align}~\label{GammaBound}
			\|\boldsymbol{d}^{(\ell)}\|_\infty \leq K_{\boldsymbol{\alpha},\boldsymbol{\gamma}}\Delta\boldsymbol{c}^{(\ell)}, \quad \ell=1,2,\dots, J,
		\end{align}
		with $K_{\boldsymbol{\alpha},\boldsymbol{\gamma}}=K_{\boldsymbol{\gamma}}\|\boldsymbol{\alpha}\|_1+K_{\boldsymbol{\alpha}}\|\boldsymbol{\gamma}\|_1$ where $K_{\boldsymbol{\alpha}} = \sum_{i\in\mathbb{Z}}|\alpha_i||i|$.
	\end{theorem}
	
	\begin{theorem}~\label{LinearStabilityTHM}
		Let $\left\{ \boldsymbol{c}^{(0)};\boldsymbol{d}^{(1)},\dots, \boldsymbol{d}^{(J)} \right\}$ and $\left\{ \widetilde{\boldsymbol{c}}^{(0)};\widetilde{\boldsymbol{d}}^{(1)},\dots, \widetilde{\boldsymbol{d}}^{(J)} \right\}$ be two pyramids of sequences. Then, there exists $L\geq 0$ such that
		\begin{align*}
			\|\boldsymbol{c}^{(J)}-\widetilde{\boldsymbol{c}}^{(J)}\|_\infty\leq L\bigg(\|\boldsymbol{c}^{(0)}-\widetilde{\boldsymbol{c}}^{(0)}\|_\infty+\sum_{i=1}^J \|\boldsymbol{d}^{(i)}-\widetilde{\boldsymbol{d}}^{(i)}\|_\infty\bigg),
		\end{align*}
		where $\boldsymbol{c}^{(J)}$ and $\widetilde{\boldsymbol{c}}^{(J)}$ are reconstructed from their respective data pyramids via~\eqref{LinearInterpolatoryReconstruction}.
	\end{theorem}
	
	
	\section{Approximated linear decimation}\label{ApproximatedDecimationSection}
	
	The multiscale transform~\eqref{LinearMultiscaleTransform} involves applying $\boldsymbol{\gamma}$, which is infinite. In this section, we develop transforms with finitely supported coefficients, which are essential in practice. We derive the decay rates of the new multiscale schemes and compare them with the decay rate of the original multiscale transform.
	
	\subsection{Truncation of the decimation coefficients}
	
	We approximate the operation $\mathcal{D}_{\boldsymbol{\gamma}}$ as defined in~\eqref{LinearDecimationScheme} by a proper truncation of $\boldsymbol{\gamma}$. Given a \emph{truncation parameter} $\varepsilon >0$, we define
	\begin{align}~\label{Truncation}
		\widetilde{\gamma}_k(\varepsilon) = \begin{cases}
			\gamma_k, & |\gamma_{k}| > \varepsilon, \\
			0, & \text{elsewhere}.
		\end{cases}
	\end{align}
	The bound~\eqref{GammaCoefficientsDecay} implies that the support of $\boldsymbol{\widetilde{\gamma}}(\varepsilon)$, which we denote by $\Omega_\varepsilon$, is finite for any $\varepsilon$. For simplicity, since we assume $\varepsilon$ is fixed, we omit the parameter from the sequence $\widetilde{\boldsymbol{\gamma}}$.
	
	The next theorem provides an upper bound for detail coefficients that are generated by~\eqref{LinearMultiscaleTransform}, but with $\mathcal{D}_{\widetilde{\boldsymbol{\gamma}}}$ as its decimation operator. Namely, given a real-valued sequence $\boldsymbol{c}^{(J)}$, $J\in\mathbb{N}$, we consider 
	\begin{align}~\label{LinearTruncatedMultiscaleTransform}
		\boldsymbol{c}^{(\ell-1)}=\mathcal{D}_{\widetilde{\boldsymbol{\gamma}}}\boldsymbol{c}^{(\ell)}, \quad \boldsymbol{d}^{(\ell)}=\boldsymbol{c}^{(\ell)}-\mathcal{S}_{\boldsymbol{\alpha}} \boldsymbol{c}^{(\ell-1)}, \quad \ell = 1,2,\dots,J.
	\end{align}
	
	\begin{theorem}~\label{LinearTruncatedDecayTHM}
		Let $\boldsymbol{c}^{(J)}$ be a real-valued sequence and denote by $\left\{\boldsymbol{c}^{(0)};\boldsymbol{d}^{(1)},\dots,\boldsymbol{d}^{(J)}\right\}$ its multiscale transform generated by~\eqref{LinearTruncatedMultiscaleTransform}. The subdivision scheme $\mathcal{S}_{\boldsymbol{\alpha}}$ is non-interpolating and $\mathcal{D}_{\widetilde{\boldsymbol{\gamma}
		}}$ is its corresponding decimation operator~\eqref{LinearDecimationScheme} with the truncated mask $\widetilde{\boldsymbol{\gamma}}$ of~\eqref{Truncation} where $\boldsymbol{\gamma}$ solves~\eqref{ConvolutionalEquation}. Then, 
		\begin{align}~\label{TruncatedGammaBound}
			\|\boldsymbol{d}^{(\ell)}\|_\infty\leq K_{\boldsymbol{\alpha},\widetilde{\boldsymbol{\gamma}}} \Delta\boldsymbol{c}^{(\ell)} + \eta\|\boldsymbol{\alpha}\|_1 \|\boldsymbol{c}^{(\ell)}\|_\infty, \quad \ell=1,2,\dots,J,
		\end{align}
		where $K_{\boldsymbol{\alpha},\widetilde{\boldsymbol{\gamma}}} = K_{\widetilde{\boldsymbol{\gamma}}}\|\boldsymbol{\alpha}\|_1 + M K_{\boldsymbol{\alpha}}$ 
		with $ \eta = \sum_{i\notin \Omega_\varepsilon} |\gamma_i|$, $M=\sum_i |\widetilde{\gamma}_i|$ and $ K_{\widetilde{\boldsymbol{\gamma}}} = 2\sum_i|\widetilde\gamma_{i}||i|$.
	\end{theorem}
	\begin{proof}
		First, we calculate a general term in $\boldsymbol{d}^{(\ell)}$. For $k\in\mathbb{Z}$ we have 
		\begin{align*}
			d^{(\ell)}_k & = c^{(\ell)}_k - \sum_i\alpha_{k-2i}\big(c^{(\ell-1)}\big)_i = \sum_i\alpha_{k-2i} \big(c^{(\ell)}_k - c^{(\ell-1)}_i\big) = \sum_i \alpha_{k-2i} \bigg(c^{(\ell)}_k -\sum_{n} \widetilde{\gamma}_{i-n}c^{(\ell)}_{2n}\bigg)\\
			& = \sum_i \alpha_{k-2i} \bigg(\sum_{n}\gamma_{i-n} c^{(\ell)}_k -\sum_{\substack{n \\ i-n\in\Omega_\varepsilon}} \gamma_{i-n}c^{(\ell)}_{2n}\bigg) \\ 
			& = \sum_i \alpha_{k-2i} \bigg(\sum_{\substack{n \\ i-n\in\Omega_\varepsilon}}\gamma_{i-n} \big(c^{(\ell)}_k-c^{(\ell)}_{2n}\big) +\sum_{\substack{n \\ i-n\notin\Omega_\varepsilon}} \gamma_{i-n}c^{(\ell)}_{k}\bigg).
		\end{align*}
		Consequently, 
		\begin{align*}
			\|\boldsymbol{d}^{(\ell)}\|_\infty & \leq \sum_i |\alpha_{k-2i}| \bigg(\sum_{\substack{n \\ i-n\in\Omega_\varepsilon}}|\gamma_{i-n}|\cdot |c^{(\ell)}_k-c^{(\ell)}_{2n}| +\sum_{\substack{n \\ i-n\notin\Omega_\varepsilon}} |\gamma_{i-n}|\cdot|c^{(\ell)}_{k}|\bigg) \\
			& \leq \sum_i |\alpha_{k-2i}| \bigg(\sum_{\substack{n \\ i-n\in\Omega_\varepsilon}}|\gamma_{i-n}|\cdot |2n-k|\cdot \Delta\boldsymbol{c}^{(\ell)} +\eta|c^{(\ell)}_{k}|\bigg) \\
			& \leq \sum_i |\alpha_{k-2i}| \bigg(\sum_{\substack{n \\ i-n\in\Omega_\varepsilon}}|\gamma_{i-n}|\cdot \big(|2n-2i|+|k-2i|\big)\cdot  \Delta\boldsymbol{c}^{(\ell)} +\eta\|\boldsymbol{c}^{(\ell)}\|_\infty\bigg) \\
			& \leq \sum_i |\alpha_{k-2i}| \bigg(\big(K_{\widetilde{\boldsymbol{\gamma}}} +M|k-2i|\big)\cdot \Delta\boldsymbol{c}^{(\ell)} + \eta\|\boldsymbol{c}^{(\ell)}\|_\infty\bigg) \\ 
			& \leq \big(K_{\widetilde{\boldsymbol{\gamma}}}\|\boldsymbol{\alpha}\|_1 + MK_{\boldsymbol{\alpha}}\big)\Delta\boldsymbol{c}^{(\ell)}+ \eta\|\boldsymbol{\alpha}\|_1\|\boldsymbol{c}^{(\ell)}\|_\infty.
		\end{align*}
	\end{proof}
	
	The term $\eta\|\boldsymbol{\alpha}\|_1\|\boldsymbol{c}^{(\ell)}\|_\infty$ in Theorem~\ref{LinearTruncatedDecayTHM} is a direct result of the truncation~\eqref{Truncation}. In particular, comparing with Theorem~\ref{LinearGammaDecayTHM}, $\varepsilon\to 0^+$ implies that  $\Omega_\varepsilon \to \mathbb{Z}, \widetilde{\boldsymbol{\gamma}} \to \boldsymbol{\gamma}, \eta\to 0, {K_{\widetilde{\boldsymbol{\gamma}}}}\to K_{\boldsymbol{\gamma}}, M\to \|\boldsymbol{\gamma}\|_1$, and thus $K_{\boldsymbol{\alpha},\widetilde{\boldsymbol{\gamma}}}\to K_{\boldsymbol{\alpha},\boldsymbol{\gamma}}$. Consequently, by substituting the terms in~\eqref{TruncatedGammaBound}, we obtain the exact bound as it appears in Theorem~\ref{LinearGammaDecayTHM}. More on the term $\eta\|\boldsymbol{\alpha}\|_1\|\boldsymbol{c}^{(\ell)}\|_\infty$, see Section~\ref{LinearNumericalExamples}.
	
	\subsection{Normalization of the truncated coefficients}
	
	Motivated by the case of manifold-valued data and following constructions of manifold-valued subdivision schemes, we require the truncated mask $\widetilde{\boldsymbol{\gamma}}$ to be shift invariant.
	
	\begin{defin} \label{def:zeta_normalization}
		Given a finitely supported mask $\widetilde{\boldsymbol{\gamma}}$ as in~\eqref{Truncation}, we define $\boldsymbol{\zeta}$ to be the following normalized mask,
		\begin{align}~\label{Normalization}
			\zeta_k=\frac{\widetilde{\gamma}_k}{\sum_{i\in\Omega_\varepsilon}\widetilde\gamma_i}, \quad k\in\mathbb{Z}.
		\end{align}
	\end{defin}
	
	As it turns out, the normalized truncated sequence $\boldsymbol{\zeta}$ of Definition~\ref{def:zeta_normalization} directly affects the decay rate of the detail coefficients. Let $\mathcal{S}_{\boldsymbol{\alpha}}$ be a subdivision scheme, and let $\mathcal{D}_{\boldsymbol{\zeta}}$ be its corresponding even-inverse decimation operator associated with the normalized truncated mask $\boldsymbol{\zeta}$ as defined in~\eqref{Normalization}. Then, we define the multiscale transform,
	\begin{align}~\label{LinearNormalizedMultiscaleTransform}
		\boldsymbol{c}^{(\ell-1)}=\mathcal{D}_{\boldsymbol{\zeta}}\boldsymbol{c}^{(\ell)}, \quad \boldsymbol{d}^{(\ell)}=\boldsymbol{c}^{(\ell)}-\mathcal{S}_{\boldsymbol{\alpha}} \boldsymbol{c}^{(\ell-1)}, \quad \ell = 1,2,\dots,J.
	\end{align}
	
	The following theorem shows that the $\sup$ norms of the detail coefficients generated by~\eqref{LinearNormalizedMultiscaleTransform} are proportional to $\Delta \boldsymbol{c}^{(\ell)}$.
	
	\begin{theorem}~\label{LinearNormalizedDecayTHM}
		Let  $\boldsymbol{c}^{(J)}$ be a real-valued sequence and denote by $\left\{\boldsymbol{c}^{(0)};\boldsymbol{d}^{(1)},\dots,\boldsymbol{d}^{(J)}\right\}$ its multiscale transform generated by~\eqref{LinearNormalizedMultiscaleTransform}. The subdivision scheme is $\mathcal{S}_{\boldsymbol{\alpha}}$ and $\mathcal{D}_{\boldsymbol{\zeta}}$ is its corresponding decimation operator with the normalized mask $\boldsymbol{\zeta}$ of~\eqref{Normalization} where $\boldsymbol{\gamma}$ solves~\eqref{ConvolutionalEquation}. Then, 
		\begin{align}~\label{NormalizedZetaBound}
			\|\boldsymbol{d}^{(\ell)}\|_\infty\leq K_{\boldsymbol{\alpha},\boldsymbol{\zeta}} \Delta\boldsymbol{c}^{(\ell)}, \quad \ell=1,2,\dots,J,
		\end{align}
		where $K_{\boldsymbol{\alpha},\boldsymbol{\zeta}} = K_{\boldsymbol{\zeta}}\|\boldsymbol{\alpha}\|_1 + M K_{\boldsymbol{\alpha}}$ 
		with $M=\sum_{i\in \Omega_\varepsilon} |\zeta_i|$ and $ K_{\boldsymbol{\zeta}} = 2\sum_{i\in\Omega_\varepsilon}|\zeta_{i}||i|$.
	\end{theorem}
	\begin{proof}
		First, we calculate a general term in $\boldsymbol{d}^{(\ell)}$. For $k\in\mathbb{Z}$ we have 
		\begin{align*}
			d^{(\ell)}_k & = c^{(\ell)}_k - \sum_i\alpha_{k-2i}c^{(\ell-1)}_i = \sum_i\alpha_{k-2i} \big(c^{(\ell)}_k - c^{(\ell-1)}_i\big) = \sum_i \alpha_{k-2i} \bigg(c^{(\ell)}_k -\sum_{n} \zeta_{i-n}c^{(\ell)}_{2n}\bigg)\\
			& = \sum_i \alpha_{k-2i} \bigg(\sum_{n}\zeta_{i-n} c^{(\ell)}_k -\sum_{n} \zeta_{i-n}c^{(\ell)}_{2n}\bigg) = \sum_i \alpha_{k-2i} \bigg(\sum_{n}\zeta_{i-n} \big(c^{(\ell)}_k-c^{(\ell)}_{2n}\big) \bigg).
		\end{align*}
		Consequently, similar arguments used in the proof of Theorem~\ref{LinearTruncatedDecayTHM} yield to
		\begin{align*}
			\|\boldsymbol{d}^{(\ell)}\|_\infty & \leq \sum_i |\alpha_{k-2i}| \bigg(K_{\boldsymbol{\zeta}} +M\cdot|k-2i|\bigg)\Delta\boldsymbol{c}^{(\ell)}\\ 
			& \leq \big(K_{\boldsymbol{\zeta}}\|\boldsymbol{\alpha}\|_1 + MK_{\boldsymbol{\alpha}}\big)\cdot \Delta\boldsymbol{c}^{(\ell)} = K_{\boldsymbol{\alpha},\boldsymbol{\zeta}}\Delta\boldsymbol{c}^{(\ell)},
		\end{align*}
		as required.
	\end{proof}
	
	To realize the importance of the normalization, we estimate the magnitude of the term $\Delta\boldsymbol{c}^{(\ell)}$ in~\eqref{NormalizedZetaBound} with respect to the level $\ell$. This is achieved by assuming a prior on the sequence $\boldsymbol{c}^{(J)}$, as the following lemma suggests.
	
	\begin{lemma}~\label{LinearDifferenceLEM}
		Let $f\colon\mathbb{R}\to\mathbb{R}$ be a differentiable, bounded real-valued function. Denote by $\boldsymbol{c}^{(J)}$, $J\in\mathbb{N}$ the function's samples over the grid $2^{-J}\mathbb{Z}$, that is, $\boldsymbol{c}^{(J)}=f\mid_{2^{-J}\mathbb{Z}}$, and let $\mathcal{D}_{\boldsymbol{\zeta}}$ be a decimation operator~\eqref{LinearDecimationScheme} associated with the shift invariant sequence $\boldsymbol{\zeta}$. Then,
		\begin{align}~\label{LinearDifferenceBound}
			\Delta\boldsymbol{c}^{(\ell)} \leq \|\boldsymbol{\zeta}\|^J_1\cdot\|f^\prime\|_\infty\cdot(2\|\boldsymbol{\zeta}\|_1)^{-\ell}, \quad \ell=0,1,\dots,J,
		\end{align}
		where $\|f^\prime\|_\infty = \sup_{x \in \mathbb{R}}|f^\prime(x)|$ and the sequences $\boldsymbol{c}^{(\ell)}$ are generated iteratively by~\eqref{LinearNormalizedMultiscaleTransform}.
	\end{lemma}
	\begin{proof}
		Since $f$ is differentiable and bounded, then by the mean value theorem, for all $k\in\mathbb{Z}$ and a fixed $J\in\mathbb{N}$, there exists $x_k$ in the open segment, which connects the parametrizations of $c^{(J)}_{k}$ and $c^{(J)}_{k+1}$, such that
		\begin{align*}
			|c^{(J)}_{k+1}-c^{(J)}_{k}|=2^{-J}|f^\prime(x_k)|,
		\end{align*}
		and by applying the supremum over all $k \in \mathbb{Z}$ we obtain $\Delta\boldsymbol{c}^{(J)}\leq 2^{-J}\|f^\prime\|_\infty$. Now, observe that for any real-valued sequence $\boldsymbol{c}$ we have
		\[ \Delta(\boldsymbol{c}\downarrow 2) \leq\sup_{k\in\mathbb{Z}}|c_{2k+2}-c_{2k}|\leq2\cdot\sup_{k\in\mathbb{Z}}|c_{k+1}-c_k|=2\Delta\boldsymbol{c},\]
		and since the convolution commutes with $\Delta$, we get 
		\begin{align}~\label{Linear_constant_P}
			\Delta\boldsymbol{c}^{(\ell-1)} &= \Delta\big(\boldsymbol{\zeta}\ast(\boldsymbol{c}^{(\ell)}\downarrow 2)\big) \\
			& \leq \|\boldsymbol{\zeta}\|_1\cdot \Delta(\boldsymbol{c}^{(\ell)}\downarrow 2) \leq 2\|\boldsymbol{\zeta}\|_1\cdot \Delta\boldsymbol{c}^{(\ell)}. \nonumber
		\end{align}
		Iterating the latter inequality starting with $\ell$ gives 
		\[ \Delta \boldsymbol{c}^{(\ell)} \leq 2\|\boldsymbol{\zeta}\|_1\cdot \Delta \boldsymbol{c}^{(\ell+1)} \leq (2\|\boldsymbol{\zeta}\|_1)^{2}\cdot \Delta \boldsymbol{c}^{(\ell+2)} \leq\cdots\leq(2\|\boldsymbol{\zeta}\|_1)^{J-\ell}\cdot \Delta \boldsymbol{c}^{(J)}, \]
		which is equivalent to~\eqref{LinearDifferenceBound}, for any $\ell=0,1,\dots,J$. 
	\end{proof}    
	
	Theorem~\ref{LinearNormalizedDecayTHM} and Lemma~\ref{LinearDifferenceLEM} illustrate the significance of the normalization~\eqref{Normalization}. Specifically, if $\boldsymbol{c}^{(J)}$ is sampled from a differentiable function, then the detail coefficients generated by the multiscale transform~\eqref{LinearNormalizedMultiscaleTransform} are bounded by a geometrically decreasing bound, as the following corollary states.
	
	\begin{coro}~\label{LinearCorollary}
		If $\boldsymbol{c}^{(J)}$ are sampled from a differentiable function $f\colon\mathbb{R}\to\mathbb{R}$ over the grid $2^{-J}\mathbb{Z}$, then the detail coefficients generated by~\eqref{LinearNormalizedMultiscaleTransform} satisfy 
		\begin{align}~\label{LinearSmoothDetails}
			\|\boldsymbol{d}^{(\ell)}\|_\infty \leq K_{\boldsymbol{\alpha},\boldsymbol{\zeta}} \|\boldsymbol{\zeta}\|^J_1\|f^\prime\|_\infty\cdot(2\|\boldsymbol{\zeta}\|_1)^{-\ell}, \quad \ell=1,2,\dots,J.
		\end{align}
		The upper bound in~\eqref{LinearSmoothDetails} decays geometrically with factor $1/(2\|\boldsymbol{\zeta}\|_1)<1/2$, with respect to the level $\ell$.
	\end{coro}
	
	A direct result from Corollary~\ref{LinearCorollary} suggests that, the detail coefficients corresponding to the highest scale can be as small as we desire. In particular, for any $\sigma>0$ one can find a sufficiently large $J_0\in\mathbb{N}$ such that $\|\boldsymbol{d}^{(J)}\|_\infty\leq \sigma$ for all $J>J_0$. A direct calculation for the minimal $J_0$ yields 
	\begin{align}~\label{MinimalJ}
		J_0=\floor{\log_2\big(\|f^\prime\|_\infty\cdot K_{\boldsymbol{\alpha},\boldsymbol{\zeta}}/\sigma\big)}+1,
	\end{align}
	in contrary to the transform~\eqref{LinearTruncatedMultiscaleTransform} where such $J_0$ is not guaranteed. Note that according to~\eqref{MinimalJ}, when $\|f^\prime\|_\infty$ is large, for example, if $f$ is rapidly changing, we need a denser grid to achieve small enough details, that is $\|\boldsymbol{d}^{(J)}\|_\infty\leq \sigma$.
	
	
	\section{Non-interpolating transform for manifold-valued data}
	
	Interpolating multiscale transforms have been studied in various nonlinear settings~\cite{grohs2009interpolatory, grohs2012definability, kaur2017contrast, lv2018novel}. In this section, we aim to adapt the non-interpolating multiscale transform~\eqref{LinearNormalizedMultiscaleTransform} to manifold-valued data.
	
	\subsection{The Riemannian analogue of a linear decimation operator}
	
	Let $\boldsymbol{\zeta}$ be a normalized finitely supported mask~\eqref{Normalization}, and let $\mathcal{M}$ be a Riemannian manifold. We extend the decimation operator~\eqref{LinearDecimationScheme} to $\mathcal{M}$ with the help of its Riemannian geodesic distance~\eqref{RiemannianDistance}, similar to the extension of the subdivision schemes, as done in Section~\ref{RiemannianAnalogue}.
	
	For any $\mathcal{M}$-valued sequence $\boldsymbol{c}$, we define
	\begin{align}~\label{ManifoldFrechetMeanDecimation}
		\mathcal{Y}_{\boldsymbol{\zeta}}(\boldsymbol{c})_{k} = \text{arg}\min_{x\in\mathcal{M}}\sum_{i\in\mathbb{Z}}\zeta_{k-i}\rho(x,c_{2i})^2, \quad k\in\mathbb{Z}.
	\end{align}
	Namely, $\mathcal{Y}_{\boldsymbol{\zeta}}(\boldsymbol{c})_{k}$ is interpreted as the Riemannian center of mass of the elements $c_{2i}$ with the corresponding weights $\zeta_{k-i}$. Moreover, it is the Riemannian analogue of the linear decimation operator $\mathcal{D}_{\boldsymbol{\zeta}}$. Results in~\cite{hardering2015intrinsic} provide conditions for the global existence of a minimizer in~\eqref{ManifoldFrechetMeanDecimation}. We formulate one such result in the following lemma.
	
	\begin{lemma}~\label{HHLemma}
		Let $\boldsymbol{\zeta}\in\mathbb{R}^m$ be a shift invariant mask, and let $\boldsymbol{c}=\left\{c_1,c_2,\dots,c_m\right\}$ be a set of $\mathcal{M}$-valued points satisfying $\rho(c_1,c_k)\leq r$ for some $r>0$ and for all $k=1,2,\dots,m$. Then, the objective function $h \colon \mathcal{M} \to \mathbb{R}$ defined as,
		\begin{align}~\label{ObjectiveFunction}
			h(x)=\sum_{i=1}^m\zeta_i\rho(x,c_{i})^2
		\end{align}
		has at least one minimum. Moreover, there exists a constant $R\leq6m\|\boldsymbol{\zeta}\|_\infty$ such that all minima of~\eqref{ObjectiveFunction} lie inside a compact ball centred around $c_1$ with radius $rR$. If $r$ is small enough, with respect to the curvature of $\mathcal{M}$, then~\eqref{ObjectiveFunction} has a unique solution.
	\end{lemma}
	
	Henceforth, we require any $\mathcal{M}$-valued admissible sequences to obey the strong conditions of Lemma~\eqref{HHLemma} for any shift invariant mask of our decimation operator $\mathcal{Y}_{\boldsymbol{\zeta}}$, that is, the Riemannian center of mass exists and unique. Moreover, we say that the decimation operator $\mathcal{Y}_{\boldsymbol{\zeta}}$ of~\eqref{ManifoldFrechetMeanDecimation} is the even-inverse of the subdivision scheme $\mathcal{T}_{\boldsymbol{\alpha}}$ of~\eqref{ManifoldFrechetMean} if its linear version $\mathcal{D}_{\boldsymbol{\zeta}}$, associated with the normalized mask $\boldsymbol{\zeta}$ of~\eqref{Normalization}, where $\boldsymbol{\gamma}$ solves~\eqref{ConvolutionalEquation}, is the even-inverse of the linear scheme $\mathcal{S}_{\boldsymbol{\alpha}}$. We illustrate the process of adaptation of the decimation operator to manifold values in the diagram of Figure~\ref{OperatorsDiagram}.
	
	\begin{figure}[H]
		\centering
		\begin{tikzcd}[sep=large,row sep=large]
			\mathcal{S}_{\boldsymbol{\alpha}} \arrow[swap, "\text{adaptation}", d, Leftrightarrow] \arrow[r,rightarrow, "\text{solving~\eqref{ConvolutionalEquation}}"] & [7ex]
			\mathcal{D}_{\boldsymbol{\gamma}} \arrow[r, "\text{truncation}"] & [7ex]
			\mathcal{D}_{\widetilde{\boldsymbol{\gamma}}} \arrow[r, "\text{normalization}"] &[7ex] \mathcal{D}_{\boldsymbol{\zeta}} \arrow[d, Leftrightarrow, "\text{ adaptation}"] \\
			\mathcal{T}_{\boldsymbol{\alpha}} \arrow[rrr,rightarrow] & & & \mathcal{Y}_{\boldsymbol{\zeta}}
		\end{tikzcd}
		\caption{The upper and lower rows represent linear and non-linear operators, respectively. The process of finding the even-inverse of the manifold-valued subdivision scheme $\mathcal{T}_{\boldsymbol{\alpha}}$ is done indirectly by going through the linear operators and adapting the resulting decimation operator.}
		\label{OperatorsDiagram}
	\end{figure}
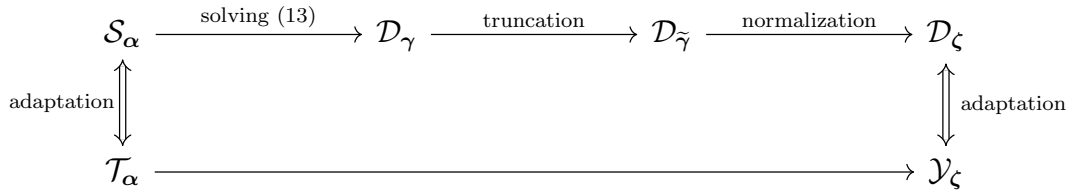
	
	\subsection{The non-interpolating multiscale transform for manifold data}
	
	Now that we have the adapted subdivision operator $\mathcal{T}_{\boldsymbol{\alpha}}$ and its corresponding decimation operator $\mathcal{Y}_{\boldsymbol{\zeta}}$ being defined, we proceed by adjusting the linear pyramid transform~\eqref{LinearNormalizedMultiscaleTransform} to manifold values. A first difference between the manifold and linear versions of the transform lies in the coefficients. For a manifold-valued transform, the sequence $\boldsymbol{c}^{(\ell)}$ at level $\ell$ is a $\mathcal{M}$-valued sequence, while the detail coefficients $\boldsymbol{d}^{(\ell)}$ are elements in the tangent bundle $T\mathcal{M}=\bigcup_{p\in\mathcal{M}}\left\{ p\right\}\times T_p\mathcal{M}$ associated with $\mathcal{M}$.
	
	Recall that in a Riemannian manifold $\mathcal{M}$, the exponential mapping $\exp_p$ maps a vector $v$ in the tangent space $T_p\mathcal{M}$ to the end point of a geodesic of length $\|v\|$, which emanates from $p\in\mathcal{M}$ with initial tangent vector $v$. Inversely, $\log_p$ is the inverse map of $\exp_p$ that takes an $\mathcal{M}$-valued element $q$ and returns a vector in the tangent space $T_p\mathcal{M}$. Following similar notations used in~\cite{grohs2009interpolatory}, we denote both maps by
	\begin{align}~\label{ManifoldOperations}
		\log_p(q)=q\ominus p, \quad\text{and}\quad \exp_p(v)=p\oplus v. 
	\end{align}
	We have thus defined the analogues $\ominus$ and $\oplus$ of the `` - '' and `` + '' operations, respectively.
	
	For any point $p\in\mathcal{M}$, we use the following notation $\ominus\colon\mathcal{M}^2\to T_p\mathcal{M}$ and $\oplus\colon\mathcal{M}\times T_p\mathcal{M}\to \mathcal{M}$. Then, the compatibility condition is
	\begin{align}  \label{CompatibilityConditions}
		(p\oplus v) \ominus p = v,
	\end{align}
	for all $v\in T_p\mathcal{M}$ within the injectivity radius of $\mathcal{M}$. That being so, we are now capable of introducing our non-interpolating multiscale transform for manifold-valued data. Let $\mathcal{M}$ be a Riemannian manifold, and let $\boldsymbol{c}^{(J)}$, $J\in\mathbb{N}$ be $\mathcal{M}$-valued sequence associated with the values over the grid $2^{-J}\mathbb{Z}$. Given a non-interpolating subdivision scheme $\mathcal{T}_{\boldsymbol{\alpha}}$ with its corresponding decimation operator $\mathcal{Y}_{\boldsymbol{\zeta}}$, we define the non-interpolating multiscale transform   
	\begin{align}~\label{ManifoldNonInterpolatingTransform}
		\boldsymbol{c}^{(\ell-1)}=\mathcal{Y}_{\boldsymbol{\zeta}}\boldsymbol{c}^{(\ell)}, \quad \boldsymbol{d}^{(\ell)}=\boldsymbol{c}^{(\ell)}\ominus\mathcal{T}_{\boldsymbol{\alpha}} \boldsymbol{c}^{(\ell-1)}, \quad \ell = 1,2,\dots,J,
	\end{align}
	where the operation $\ominus$ is the $\log$ map~\eqref{ManifoldOperations} associated with $\mathcal{M}$.
	
	Indeed, the transform~\eqref{ManifoldNonInterpolatingTransform} is the non-linear analogue of~\eqref{LinearNormalizedMultiscaleTransform}. The process~\eqref{ManifoldNonInterpolatingTransform} yields a pyramid of sequences $\left\{\boldsymbol{c}^{(0)};\boldsymbol{d}^{(1)},\dots,\boldsymbol{d}^{(J)}\right\}$ where $\boldsymbol{c}^{(0)}$ are the coarse approximation coefficients given over the integers, and $\boldsymbol{d}^{(\ell)}$, $\ell=1,2,\dots,J$ are the detail coefficients at level $\ell$ given over the values of the grids $2^{-\ell}\mathbb{Z}$, respectively. By the construction of~\eqref{ManifoldNonInterpolatingTransform} we verify that
	$\boldsymbol{c}^{(0)}$ is an $\mathcal{M}$-valued sequence and the elements of  $\boldsymbol{d}^{(\ell)}$ lie in $T\mathcal{M}$ for all $\ell=1,2,\dots,J$. Then, we obtain synthesis by the iterations,
	\begin{align}~\label{ManifoldNonInterpolatoryReconstruction}
		\boldsymbol{c}^{(\ell)}=\mathcal{T}_{\boldsymbol{\alpha}} \boldsymbol{c}^{(\ell-1)}\oplus\boldsymbol{d}^{(\ell)}, \quad \ell=1,2,\dots,J.
	\end{align}
	The synthesis~\eqref{ManifoldNonInterpolatoryReconstruction} is the analogue of~\eqref{LinearInterpolatoryReconstruction}, and it is the inverse transform of~\eqref{ManifoldNonInterpolatingTransform}.
	
	
	\section{Coefficients decay and reconstruction stability}
	
	A vital feature of any multiscale transform is the rate at which the detail coefficients become small and, therefore, from a certain point, negligible. This feature is also the basis of various thresholding techniques for compression, smoothness analysis, and denoising, as we will see in Section~\ref{NumericalExamples}. This section shows that under stability of subdivision schemes, the detail coefficients are bounded, and the inverse transform is stable.
	
	\subsection{Displacement safe operators}
	
	We proceed with a series of useful definitions. Firstly, for admissible  $\mathcal{M}$-valued sequences $\boldsymbol{c}$ and $\boldsymbol{m}$, we denote the following. Let
	\[  \Delta_\mathcal{M}(\boldsymbol{c})=\sup_{k\in\mathbb{Z}}\rho(c_{k+1},c_{k}) \quad \text{and} \quad \mu(\boldsymbol{c},\boldsymbol{m})=\sup_{k\in\mathbb{Z}}\rho(c_k,m_k) ,\] be the supremum distance between consecutive elements in $\boldsymbol{c}$ and the corresponding elements of the sequences $\boldsymbol{c}$ and $\boldsymbol{m}$, respectively. We continue with the definition of displacement-safe subdivision schemes, as introduced in~\cite{dyn2017manifold}. This condition means the refinement operator generates points in a controlled fashion regarding the data points' distances. In linear schemes, for example, this condition is automatically satisfied.
	\begin{defin}~\label{Displacement-Safe}
		We say that the subdivision operator $\mathcal{T}$ is \emph{displacement-safe} if there exists $E_\mathcal{T}\geq0$ such that 
		$\mu(\boldsymbol{c},\mathcal{T}(\boldsymbol{c})\downarrow 2)\leq E_\mathcal{T}\Delta_\mathcal{M}(\boldsymbol{c})$
		for any $\mathcal{M}$-valued sequence $\boldsymbol{c}$.
	\end{defin}
	
	Displacement-safety plays a significant role for the convergence analysis of non-interpolating subdivision schemes over manifolds~\cite{dyn2017manifold, SvenjaWallnerConvergenceSphere}. Next, we introduce a new condition analogous to Definition~\ref{Displacement-Safe} , but for decimation operators.
	\begin{defin}~\label{Decimation-Safe}
		We say that the decimation operator $\mathcal{Y}$ is \emph{decimation-safe} if there exists $F_\mathcal{Y}\geq0$ such that 
		$\mu(\mathcal{Y}\boldsymbol{c},\boldsymbol{c}\downarrow 2)\leq F_\mathcal{Y}\Delta_\mathcal{M}(\boldsymbol{c})$
		for any $\mathcal{M}$-valued sequence $\boldsymbol{c}$.
	\end{defin}
	
	To simplify the terminologies, we say that the pair $\big(\mathcal{T}_{\boldsymbol{\alpha}},\mathcal{Y}_{\boldsymbol{\zeta}}\big)$ is \emph{safe} if simultaneously both the operator $\mathcal{T}_{\boldsymbol{\alpha}}$ is displacement-safe and $\mathcal{Y}_{\boldsymbol{\zeta}}$ is its corresponding decimation-safe decimation operator. With the definitions being stated, we note that the linear pyramid transforms~\eqref{LinearTruncatedMultiscaleTransform} and~\eqref{LinearNormalizedMultiscaleTransform} involve safe pairs of subdivision and decimation operators, as we state and prove in the following lemma.
	
	\begin{lemma}
		In the Euclidean case, the linear pair $\big(\mathcal{S}_{\boldsymbol{\alpha}},\mathcal{D}_{\boldsymbol{\zeta}}\big)$ is safe.
	\end{lemma}
	\begin{proof}
		The proof that $\mathcal{S}_{\boldsymbol{\alpha}}$ is displacement safe appears in~\cite{dyn2017manifold}, we proceed with proving that linear decimation operators are decimation-safe. For any real-valued sequence $\boldsymbol{c}$ and index $k\in\mathbb{Z}$ we have
		\begin{align*}
			|(\mathcal{D}_{\boldsymbol{\zeta}} \boldsymbol{c})_k-(\boldsymbol{c}\downarrow 2)_k|&=\big|\sum_{i\in\mathbb{Z}}{\zeta}_{k-i}c_{2i} - c_{2k}\big|=\big|\sum_{i\in\mathbb{Z}}\zeta_{k-i}c_{2i} - \sum_{i\in\mathbb{Z}}\zeta_{k-i}c_{2k}\big| \\
			& =\big|\sum_{i\in\mathbb{Z}}\zeta_{k-i}\big(c_{2i} - c_{2k}\big)\big|\leq 2\sum_{i\in\mathbb{Z}}|\zeta_{k-i}|\cdot|k-i|\cdot \Delta \boldsymbol{c}.
		\end{align*}
		Now, applying $\sup_{k\in\mathbb{Z}}$ on both sides gives
		$$\|\mathcal{D}_{\boldsymbol{\zeta}}\boldsymbol{c}-\boldsymbol{c}\downarrow 2\|_\infty\leq F_{\mathcal{D}} \Delta \boldsymbol{c},$$
		with $F_{\mathcal{D}}=\sup_{k\in\mathbb{Z}}2\sum_{i\in\mathbb{Z}}|\zeta_{k-i}|\cdot|k-i|<\infty$ since $\boldsymbol{\zeta}$ has a compact support. 
	\end{proof}
	
	Note that in the case where $\mathcal{T}_{\boldsymbol{\alpha}}$ is interpolating, the corresponding decimation operator $\mathcal{Y}_{\boldsymbol{\zeta}}$ is simply the downsampling operator $\downarrow 2$, see e.g.,~\cite{grohs2010stability, grohs2009interpolatory}. Thus, the pair $\big(\mathcal{T}_{\boldsymbol{\alpha}},\mathcal{Y}_{\boldsymbol{\zeta}}\big)$ is safe. In particular, $E_\mathcal{T}=F_\mathcal{Y}=0$. Otherwise, Lemma~\ref{HHLemma} guarantees that the pair $\big(\mathcal{T}_{\boldsymbol{\alpha}},\mathcal{Y}_{\boldsymbol{\zeta}}\big)$ is safe, as the following proposition suggests.
	
	\begin{prop}~\label{Safe_Pair}
		Let $\mathcal{T}_{\boldsymbol{\alpha}}$ be a non-interpolating subdivision scheme over $\mathcal{M}$, the pair $\big(\mathcal{T}_{\boldsymbol{\alpha}},\mathcal{Y}_{\boldsymbol{\zeta}}\big)$ is safe for any truncation parameter $\varepsilon$ involved in determining the shift invariant mask $\boldsymbol{\zeta}$.
	\end{prop}
	
	\begin{proof}
		Here we only prove that $\mathcal{Y}_{\boldsymbol{\zeta}}$ is decimation safe, see Definition~\ref{Decimation-Safe}, the proof that $\mathcal{T}_{\boldsymbol{\alpha}}$ is displacement safe is similar. Let $\boldsymbol{c}$ be an admissible $\mathcal{M}$-valued sequence satisfying
		$$\mathcal{H}=\frac{\sup_{|i-j|\leq |\Omega_\varepsilon| }\rho(c_i,c_j)}{\Delta_{\mathcal{M}}(\boldsymbol{c})}<\infty,$$
		where $\Omega_\varepsilon$ is the compact support of $\boldsymbol{\zeta}$, as in~\eqref{Truncation}. We express $\boldsymbol{c}$ as a countable union of pairwise overlapping sets $\Xi_k$, $k\in\mathbb{Z}$, where $\Xi_k$ consists of all the elements of $\boldsymbol{c}$ involved in calculating $\mathcal{Y}_{\boldsymbol{\zeta}}(\boldsymbol{c})_k$, see~\eqref{ManifoldFrechetMeanDecimation}. Indeed, $\Xi_k$ contains $|\Omega_\varepsilon|$-many elements including $c_{2k}$. Lemma~\ref{HHLemma} guarantees $$\rho(\mathcal{Y}_{\boldsymbol{\zeta}}(\boldsymbol{c})_k, c_{2k})\leq 6|\Omega_\varepsilon|\|\boldsymbol{\zeta}\|_\infty \cdot r_k,$$
		where $r_k = \max_{c_{j}\in\Xi_k}\rho (c_{2k},c_{j})$. By the definition of $\mathcal{H}$, we have $r_k\leq \mathcal{H}\Delta_{\mathcal{M}}(\boldsymbol{c})$, $k\in\mathbb{Z}$. Thus, applying $\sup_{k \in \mathbb{Z}}$ yields $\mu(\mathcal{Y}_{\boldsymbol{\zeta}}, \boldsymbol{c}\downarrow 2)\leq F_{\mathcal{Y}}\Delta_{\mathcal{M}}(\boldsymbol{c})$, as required, with $F_{\mathcal{Y}} = 6|\Omega_\varepsilon|\|\boldsymbol{\zeta}\|_\infty\mathcal{H}$.
	\end{proof}
	
	In the same manner of using Lemma~\ref{HHLemma} to prove Proposition~\ref{Safe_Pair}, the former guarantees that for any admissible $\mathcal{M}$-valued sequence $\boldsymbol{c}$, we have 
	\begin{align}~\label{Q_constant}
		\Delta_{\mathcal{M}}(\mathcal{T}_{\boldsymbol{\alpha}}\boldsymbol{c}) \leq Q\Delta_{\mathcal{M}}(\boldsymbol{c})
	\end{align}
	for some $Q>0$. Estimate~\eqref{Q_constant} is used in the proof of Proposition~\ref{Q_proposition} in the next section. Later on, in the proof of Lemma~\ref{ManifoldDecayLEM}, we will observe the analogue of~\eqref{Q_constant}, but for decimation operators.
	
	\subsection{Decay of detail coefficients}
	
	We proceed with defining a stable subdivision rule.
	
	\begin{defin}~\label{Stability}
		We say that the subdivision operator $\mathcal{T}$ is stable if there exists $S_\mathcal{T}\geq0$ such that
		$\mu(\mathcal{T}\boldsymbol{c}, \mathcal{T}\boldsymbol{m})\leq S_\mathcal{T}\mu(\boldsymbol{c},\boldsymbol{m})$
		for all $\mathcal{M}$-valued sequences $\boldsymbol{c}$ and $\boldsymbol{m}$.
	\end{defin}
	Stable subdivision schemes have been studied in~\cite{grohs2010stability}. The next proposition helps in proving the following theorem, which is the analogue to Theorem~\ref{LinearNormalizedDecayTHM}.
	\begin{prop}~\label{Q_proposition}
		For any admissible sequence $\boldsymbol{c}$, we have $$\mu\big(\boldsymbol{c}, \mathcal{T}_{\boldsymbol{\alpha}}(\boldsymbol{c} \downarrow2)\big) \leq (1 + 2E_{\mathcal{T}} + Q)\Delta_\mathcal{M}(\boldsymbol{c}),$$
		where the constants $E_{\mathcal{T}}$ and $Q$ are from Definition~\ref{Displacement-Safe} and~\eqref{Q_constant}, respectively.
	\end{prop}
	\begin{proof}
		Recall that $\mu\big(\boldsymbol{c}, \mathcal{T}_{\boldsymbol{\alpha}}(\boldsymbol{c} \downarrow2)\big)=\sup_{k \in \mathbb{Z}}\rho\big(c_k, \mathcal{T}_{\boldsymbol{\alpha}}(\boldsymbol{c} \downarrow2)_k\big)$. If $k=2j$ is even, then by the
		displacement-safe inequality, see Definition~\ref{Displacement-Safe}, we have
		$$\rho\big(c_{2j}, \mathcal{T}_{\boldsymbol{\alpha}}(\boldsymbol{c} \downarrow2)_{2j}\big) = \rho \big((\boldsymbol{c}\downarrow2)_{j}, \mathcal{T}_{\boldsymbol{\alpha}}(\boldsymbol{c} \downarrow2)_{2j}\big)\leq E_{\mathcal{T}}\Delta_{\mathcal{M}}(\boldsymbol{c}\downarrow2)\leq 2E_{\mathcal{T}}\Delta_{\mathcal{M}}(\boldsymbol{c}).$$
		Otherwise, we have
		\begin{align*}
			\rho\big(c_{2j+1}, \mathcal{T}_{\boldsymbol{\alpha}}(\boldsymbol{c} \downarrow2)_{2j+1}\big) & \leq \rho(c_{2j+1}, c_{2j}) + \rho(c_{2j}, \mathcal{T}_{\boldsymbol{\alpha}}(\boldsymbol{c} \downarrow2)_{2j}) + \rho(\mathcal{T}_{\boldsymbol{\alpha}}(\boldsymbol{c} \downarrow2)_{2j}, \mathcal{T}_{\boldsymbol{\alpha}}(\boldsymbol{c} \downarrow2)_{2j+1}) \\ & \leq (1 + 2E_{\mathcal{T}} + Q)\Delta_{\mathcal{M}}(\boldsymbol{c}),
		\end{align*}
		as required.
	\end{proof}
	\begin{theorem}~\label{DecayTheorem}
		Let $\mathcal{T}_{\boldsymbol{\alpha}}$ be a stable subdivision scheme, then there exists $K\geq 0$ such that the detail coefficients generated by~\eqref{ManifoldNonInterpolatingTransform} satisfy
		\begin{align}~\label{DetailsBound1}
			\|\boldsymbol{d}^{(\ell)}\|_\infty\leq K\Delta_\mathcal{M}(\boldsymbol{c}^{(\ell)}),\quad  \ell=1,2,\dots, J,
		\end{align}
		where $\|\boldsymbol{d}^{(\ell)}\|_\infty = \sup_{k \in \mathbb{Z}}\|d^{(\ell)}_k\|$.
	\end{theorem}
	\begin{proof}
		First, observe that $\|\boldsymbol{d}^{(\ell)}\|_\infty=\|\boldsymbol{c}^{(\ell)}\ominus\mathcal{T}_{\boldsymbol{\alpha}}\boldsymbol{c}^{(\ell-1)}\|_\infty = \mu(\boldsymbol{c}^{(\ell)}, \mathcal{T}_{\boldsymbol{\alpha}}\boldsymbol{c}^{(\ell-1)})$. Since the pair $\big(\mathcal{T}_{\boldsymbol{\alpha}},\mathcal{Y}_{\boldsymbol{\zeta}}\big)$ is safe, see Proposition~\ref{Safe_Pair}, then by the triangle inequality we have
		\begin{align*}
			\mu(\boldsymbol{c}^{(\ell)}, \mathcal{T}_{\boldsymbol{\alpha}}\boldsymbol{c}^{(\ell-1)}) & \leq \mu\big(\boldsymbol{c}^{(\ell)}, \mathcal{T}_{\boldsymbol{\alpha}}(\boldsymbol{c}^{(\ell)}\downarrow 2)\big)+\mu\big(\mathcal{T}_{\boldsymbol{\alpha}}(\boldsymbol{c}^{(\ell)}\downarrow 2), \mathcal{T}_{\boldsymbol{\alpha}}\mathcal{Y}_{\boldsymbol{\zeta}}\boldsymbol{c}^{(\ell)}\big) \\
			& \leq (1 + 2E_{\mathcal{T}} + Q)\Delta_\mathcal{M}(\boldsymbol{c}^{(\ell)})+S_{\mathcal{T}}\mu\big(\boldsymbol{c}^{(\ell)}\downarrow 2, \mathcal{Y}_{\boldsymbol{\zeta}}\boldsymbol{c}^{(\ell)}\big) \\ & \leq (1 + 2E_{\mathcal{T}} + Q + S_{\mathcal{T}}F_{\mathcal{Y}})\Delta_\mathcal{M}(\boldsymbol{c}^{(\ell)}),
		\end{align*}
		as required where $K=1 + 2E_{\mathcal{T}} + Q + S_{\mathcal{T}}F_{\mathcal{Y}}$.
	\end{proof}
	
	To proceed, we estimate the magnitude of $\Delta_\mathcal{M} (\boldsymbol{c}^{(\ell)})$ in~\eqref{DetailsBound1} by assuming a prior on the admissible data points $\boldsymbol{c}^{(J)}$. Recall that if $\Gamma$ is an $\mathcal{M}$-valued regular differentiable curve, then $\nabla \Gamma(x)$ denotes the intrinsic gradient of $\Gamma$ at point $x\in\mathcal{M}$, i.e., the velocity vector of $\Gamma$ at point $x\in\mathcal{M}$ lying in $T_x\mathcal{M}$~\cite{carmo1992riemannian}. The following lemma is the analogue to Lemma~\ref{LinearDifferenceLEM}.
	
	\begin{lemma}~\label{ManifoldDecayLEM}	
		Let $\Gamma$ be a regular differentiable curve over $\mathcal{M}$. Denote by $\boldsymbol{c}^{(J)}$, $J\in\mathbb{N}$ the curve's samples over the arc-length parametrization grid $2^{-J}\mathbb{Z}$, that is, $\boldsymbol{c}^{(J)}=\Gamma\mid_{2^{-J}\mathbb{Z}}$, and let $\mathcal{Y}_{\boldsymbol{\zeta}}$ be a decimation operator associated with the shift invariant mask $\boldsymbol{\zeta}$. Then, there exists $P>1$, depending on $\boldsymbol{\zeta}$ and the curvature of $\mathcal{M}$, such that 
		\begin{align}~\label{P_proposition}
			\Delta_{\mathcal{M}}(\boldsymbol{c}^{(\ell)}) \leq P^J\|\nabla \Gamma\|_\infty\cdot(2P)^{-\ell}, \quad \ell=0,1,\dots,J,
		\end{align}
		where $\|\nabla \Gamma\|_\infty = \sup_s\|\nabla \Gamma(s)\|$ and the sequences $\boldsymbol{c}^{(\ell)}$ are generated recursively by~\eqref{ManifoldNonInterpolatingTransform}.
	\end{lemma}
	
	\begin{proof}
		First, since $\mathcal{Y}_{\boldsymbol{\zeta}}$ is decimation safe, see Proposition~\ref{Safe_Pair}, then for $\ell=1,2,\dots,J$ and $k\in\mathbb{Z}$,
		\begin{align*}
			\rho\big((\mathcal{Y}_{\boldsymbol{\zeta}} \boldsymbol{c}^{(\ell)})_{k+1}, (\mathcal{Y}_{\boldsymbol{\zeta}} \boldsymbol{c}^{(\ell)})_k \big) \leq \rho\big((\mathcal{Y}_{\boldsymbol{\zeta}} \boldsymbol{c}^{(\ell)})_{k+1}, c_{2k+2} \big) + \rho\big( c_{2k+2}, c_{2k} \big) + \rho\big( c_{2k}, (\mathcal{Y}_{\boldsymbol{\zeta}} \boldsymbol{c}^{(\ell)})_k \big).
		\end{align*}
		Applying $\sup_{k \in \mathbb{Z}}$ yields to
		\begin{align*}
			\Delta_{\mathcal{M}}(\mathcal{Y}_{\boldsymbol{\zeta}}\boldsymbol{c}^{(\ell)}) \leq 2F_{\mathcal{Y}}\Delta_{\mathcal{M}}(\boldsymbol{c}^{(\ell)}) + 2\Delta_{\mathcal{M}}(\boldsymbol{c}^{(\ell)})= 2P\cdot \Delta_{\mathcal{M}}(\boldsymbol{c}^{(\ell)}),
		\end{align*}
		where $P = 1+F_\mathcal{Y}$. In other words, $\Delta_{\mathcal{M}}(\boldsymbol{c}^{(\ell-1)})\leq 2P\cdot \Delta_{\mathcal{M}}(\boldsymbol{c}^{(\ell)})$. Iteratively, we get
		\begin{align*}
			\Delta_{\mathcal{M}}(\boldsymbol{c}^{(\ell)})\leq (2P)^{J-\ell} \Delta_{\mathcal{M}}(\boldsymbol{c}^{(J)}), \quad \ell =0,1,\dots, J.
		\end{align*}
		Now, since the sequence $\boldsymbol{c}^{(J)}$ is sampled from a regular differentiable curve $\Gamma$ over the arc-length parametrization grid $2^{-J}\mathbb{Z}$, we immediately have $\Delta_{\mathcal{M}}(\boldsymbol{c}^{(J)})\leq \|\nabla \Gamma\|_\infty 2^{-J}$. In total, inequality~\eqref{P_proposition} follows.
	\end{proof}
	
	Inequality~\eqref{P_proposition} is the analogue to the estimate in~\eqref{LinearDifferenceBound}. The next corollary is the analogue of Corollary~\ref{LinearCorollary}, it shows that if $\boldsymbol{c}^{(J)}$ is sampled from a differentiable curve over the manifold, then the detail coefficients generated by~\eqref{ManifoldNonInterpolatingTransform} are bounded by a geometrically decreasing bound with factor $1/(2P)<1/2$.
	
	\begin{coro}~\label{ManifoldCorollary}
		Let $\boldsymbol{c}^{(J)}$ denote the samples of a regular differentiable curve $\Gamma$ over $\mathcal{M}$ on the equispaced arc-length parameterized grid $2^{-J}\mathbb{Z}$, where $\|\nabla \Gamma\|_\infty$ assumed to be finite. Let $\mathcal{T}_{\boldsymbol{\alpha}}$ be stable subdivision scheme, then the detail coefficients generated by~\eqref{ManifoldNonInterpolatingTransform} satisfy
		\begin{align}~\label{ManifoldDecay}
			\|\boldsymbol{d}^{(\ell)}\|_\infty\leq KP^J \|\nabla \Gamma\|_\infty\cdot(2P)^{-\ell}, \quad \ell =1,2,\dots, J,
		\end{align} 
		with the same values of $K$ in~\eqref{DetailsBound1} and $P$ in~\eqref{P_proposition}. The upper bound in~\eqref{ManifoldDecay} decays geometrically with factor $1/(2P)<1/2$, with respect to the level $\ell$.
	\end{coro}
	
	In inequality~\eqref{ManifoldDecay}, if $\ell=J$ then $\|\boldsymbol{d}^{(J)}\|_\infty\leq K\|\nabla \Gamma\|_\infty \cdot 2^{-J}$. Thus, Corollary~\ref{ManifoldCorollary} guarantees that for any $\sigma>0$ there exists some $J_0\in\mathbb{N}$ such that $\|\boldsymbol{d}^{(J)}\|_\infty\leq\sigma$ for all $J>J_0$. The minimal $J_0$ is similar in its form to~\eqref{MinimalJ}.
	
	\subsection{Stability of inverse transform}
	
	In this section, we discuss the stability of the inverse multiscale transform~\eqref{ManifoldNonInterpolatoryReconstruction}. As a first conclusion, we show that the stability of the subdivision operator $\mathcal{T}_{\boldsymbol{\alpha}}$, as presented in Definition~\ref{Stability}, induces a stability result for the inverse transform.
	
	\begin{theorem}~\label{ManifoldStabilityTHM1}
		Let $\left\{ \boldsymbol{c}^{(0)};\boldsymbol{d}^{(1)},\dots, \boldsymbol{d}^{(J)} \right\}$ and $\left\{ \widetilde{\boldsymbol{c}}^{(0)};\widetilde{\boldsymbol{d}}^{(1)},\dots, \widetilde{\boldsymbol{d}}^{(J)} \right\}$ be two pyramids of sequences. Let $\mathcal{T}_{\boldsymbol{\alpha}}$ be stable with constant $S_{\mathcal{T}}$, as in Definition~\ref{Stability}. Then, the synthesis sequences $\boldsymbol{c}^{(J)}$ and $\widetilde{\boldsymbol{c}}^{(J)}$, which are reconstructed from the above two data pyramids via~\eqref{ManifoldNonInterpolatoryReconstruction}, satisfy 
		\begin{align}~\label{ManifoldStability}
			\mu(\boldsymbol{c}^{(J)},\widetilde{\boldsymbol{c}}^{(J)})\leq L\bigg( \mu(\boldsymbol{c}^{(0)},\widetilde{\boldsymbol{c}}^{(0)}) +\sum_{i=1}^{J} \|\boldsymbol{d}^{(i)}\|_\infty + \|\widetilde{\boldsymbol{d}}^{(i)}\|_\infty \bigg),
		\end{align}
		with $L=1$ if $S_{\mathcal{T}}\leq1$, and $L=S_{\mathcal{T}}^J$ otherwise.
	\end{theorem}
	\begin{proof}
		Recall that for any point $p\in\mathcal{M}$, by~\eqref{CompatibilityConditions} we have that 
		$\rho (p\oplus v, p)=\| v \|$ with the Euclidean norm and for all $v\in T_p\mathcal{M}$ within the injectivity radius of $\mathcal{M}$. In other words, the projection of vector $v$ that lies in the tangent space of base point $p$, has a length of $\|v\|_\infty$. Now, observe that
		\begin{align*}
			\mu(\boldsymbol{c}^{(J)},\widetilde{\boldsymbol{c}}^{(J)})&=\mu(\boldsymbol{c}^{(J)},\mathcal{T}_{\boldsymbol{\alpha}}\boldsymbol{c}^{(J-1)})+\mu(\mathcal{T}_{\boldsymbol{\alpha}}\boldsymbol{c}^{(J-1)},\mathcal{T}_{\boldsymbol{\alpha}}\widetilde{\boldsymbol{c}}^{(J-1)})+\mu(\mathcal{T}_{\boldsymbol{\alpha}}\widetilde{\boldsymbol{c}}^{(J-1)},\widetilde{\boldsymbol{c}}^{(J)}) \\ & \leq 
			\|\boldsymbol{d}^{(J)}\|_\infty + S_{\mathcal{T}}\mu(\boldsymbol{c}^{(J-1)},\widetilde{\boldsymbol{c}}^{(J-1)}) + \|\widetilde{\boldsymbol{d}}^{(J)}\|_\infty.
		\end{align*}
		Iterating the latter triangle inequality for the middle term gives
		\begin{align*}
			\mu(\boldsymbol{c}^{(J)},\widetilde{\boldsymbol{c}}^{(J)})\leq \|\boldsymbol{d}^{(J)}\|_\infty + \|\boldsymbol{d}^{(J-1)}\|_\infty + S_{\mathcal{T}}^2\mu(\boldsymbol{c}^{(J-2)},\widetilde{\boldsymbol{c}}^{(J-2)}) + \|\widetilde{\boldsymbol{d}}^{(J-1)}\|_\infty + \|\widetilde{\boldsymbol{d}}^{(J)}\|_\infty,
		\end{align*}
		which inductively yields to 
		\begin{align*}
			\mu(\boldsymbol{c}^{(J)},\widetilde{\boldsymbol{c}}^{(J)})\leq \sum_{i=1}^J \|\boldsymbol{d}^{(i)}\|_\infty + S_{\mathcal{T}}^J\mu(\boldsymbol{c}^{(0)},\widetilde{\boldsymbol{c}}^{(0)}) + \sum_{i=1}^J \|\widetilde{\boldsymbol{d}}^{(i)}\|_\infty.
		\end{align*}
		The required is thus obtained with $L=1$ if $S_{\mathcal{T}}\leq 1$, and $L=S^J_{\mathcal{T}}$ otherwise.
	\end{proof}
	
	Note that if the two pyramids in Theorem~\ref{ManifoldStabilityTHM1} were generated by~\eqref{ManifoldNonInterpolatingTransform} to represent samples of two differentiable curves $\Gamma$ and $\widetilde{\Gamma}$, respectively. Then, by making use of Corollary~\ref{ManifoldCorollary}, the sum term in~\eqref{ManifoldStability} can be bounded in terms of the constants $K$, $P$, $\|\nabla \Gamma\|_\infty$ and $\|\nabla \widetilde{\Gamma}\|_\infty$, which depend on the geometry of $\mathcal{M}$.
	
	A special case, where stability in the spirit of Theorem~\ref{LinearStabilityTHM} can be obtained intrinsically, is when the curvature of the manifold is bounded. Next, we present such a result, assuming $\mathcal{M}$ is complete, open manifold with non-negative sectional curvature. For that, we recall two classical theorems: the first and second Rauch comparison theorems (the second is actually due to Berger), tailored to our settings and notation. For more details, see~\cite[Chapter 3]{gromoll2009metric} and references therein.
	
	We use the following notation. Denote by $p_k \in \mathcal{M}$ two points, $k=1,2$, and $v_k \in T_{p_k}\mathcal{M}$ their vectors in the tangent spaces such that $\| v_1 \|=\| v_2 \|$ and the value is smaller than the injectivity radius of $\mathcal{M}$. Let $ G(p_1,p_2)$ be the geodesic line connecting $p_1$ and $p_2$ and $PG_{p_2}(v_1) \in T_{p_2}\mathcal{M}$ be the parallel transport of $v_1$ along $G(p_1,p_2)$ to $T_{p_2}\mathcal{M}$. Then, the first Rauch theorem suggests that 
	\begin{equation} \label{eqn:rauch1}
		\rho \big(p_2 \oplus v_2, p_2 \oplus \pg_{p_2}(v_1)\big)  \le \|v_2- \pg_{p_2}(v_1) \|.
	\end{equation}
	Moreover, the second Rauch theorem implies that 
	\begin{equation} \label{eqn:rauch2}
		\rho \big(p_1 \oplus v_1, p_2 \oplus \pg_{p_2}(v_1)\big)  \le \rho( p_1, p_2).
	\end{equation}
	We are ready for the stability conclusion.
	
	\begin{theorem}~\label{ManifoldStabilityTHM}
		Let $\mathcal{M}$ be a complete, open manifold with non-negative sectional curvature. Denote by $\left\{ \boldsymbol{c}^{(0)};\boldsymbol{d}^{(1)},\dots, \boldsymbol{d}^{(J)} \right\}$ and $\left\{ \widetilde{\boldsymbol{c}}^{(0)};\widetilde{\boldsymbol{d}}^{(1)},\dots, \widetilde{\boldsymbol{d}}^{(J)} \right\}$ two pyramids of sequences such that $ \| {d}_k^{(\ell)} \|  = \| \widetilde{d}_k^{(\ell)} \|$ and with values smaller than the injectivity radius of $\mathcal{M}$, for all $\ell=1,\ldots,J$ and $k\in \mathbb{Z}$. Also, assume that $\mu(\boldsymbol{c}^{(0)},\widetilde{\boldsymbol{c}}^{(0)})$ is sufficiently small so geodesics exist between all pairs ${c}_k^{(\ell)},\widetilde{{c}}_k^{(\ell)}$ for all $\ell=1,\ldots,J$ and $k\in \mathbb{Z}$, as reconstructed from the above two data pyramids via~\eqref{ManifoldNonInterpolatoryReconstruction}. Assume $\mathcal{T}_{\boldsymbol{\alpha}}$ is stable with constant $S_{\mathcal{T}}$, as in Definition~\ref{Stability}. 
		Then, the synthesis sequences $\boldsymbol{c}^{(J)}$ and $\widetilde{\boldsymbol{c}}^{(J)}$ satisfy 
		\begin{align} \label{eqn:stability}
			\mu( \boldsymbol{c}^{(J)},\widetilde{\boldsymbol{c}}^{(J)} )\leq L\bigg(\mu(\boldsymbol{c}^{(0)},\widetilde{\boldsymbol{c}}^{(0)}) +\sum_{i=1}^J \| \widehat{\boldsymbol{d}}^{(i)}-\widetilde{\boldsymbol{d}}^{(i)}\|_\infty\bigg),
		\end{align}
		where $\widehat{d}_k^{(i)} = \pg_{ (\mathcal{T}_{\boldsymbol{\alpha}}\widetilde{\boldsymbol{c}}^{(i-1)})_k} (d_k^{(i)})$, with $L=1$ if $S_{\mathcal{T}}\leq1$, and $L=S_{\mathcal{T}}^J$ otherwise.
	\end{theorem}
	\begin{proof}
		Observe that
		\begin{align*}
			\mu(\boldsymbol{c}^{(J)},\widetilde{\boldsymbol{c}}^{(J)})
			& \leq
			\mu\big( \boldsymbol{c}^{(J)}, \mathcal{T}_{\boldsymbol{\alpha}} \widetilde{\boldsymbol{c}}^{(J-1)}\oplus  \widehat{\boldsymbol{d}}^{(J)}\big) + \mu\big(\mathcal{T}_{\boldsymbol{\alpha}} \widetilde{\boldsymbol{c}}^{(J-1)}\oplus  \widehat{\boldsymbol{d}}^{(J)},\widetilde{\boldsymbol{c}}^{(J)}\big)
			\\ & =
			\mu\big( \mathcal{T}_{\boldsymbol{\alpha}} {\boldsymbol{c}}^{(J-1)}\oplus \boldsymbol{d}^{(J)},
			\mathcal{T}_{\boldsymbol{\alpha}} \widetilde{\boldsymbol{c}}^{(J-1)}\oplus \widehat{\boldsymbol{d}}^{(J)} \big) +
			\mu\big(\mathcal{T}_{\boldsymbol{\alpha}} \widetilde{\boldsymbol{c}}^{(J-1)}\oplus \widehat{\boldsymbol{d}}^{(J)}, \mathcal{T}_{\boldsymbol{\alpha}} \widetilde{\boldsymbol{c}}^{(J-1)}\oplus  \widetilde{\boldsymbol{d}}^{(J)} \big) 
			\\ & \leq		
			\mu( \mathcal{T}_{\boldsymbol{\alpha}} {\boldsymbol{c}}^{(J-1)}, 	\mathcal{T}_{\boldsymbol{\alpha}} \widetilde{\boldsymbol{c}}^{(J-1)} ) + \| \widehat{\boldsymbol{d}}^{(J)}-\widetilde{\boldsymbol{d}}^{(J)}\|_\infty
			\\ & \leq		
			S_{\mathcal{T}} \mu( {\boldsymbol{c}}^{(J-1)},  \widetilde{\boldsymbol{c}}^{(J-1)} ) + \| \widehat{\boldsymbol{d}}^{(J)}-\widetilde{\boldsymbol{d}}^{(J)}\|_\infty.
		\end{align*}
		For the first inequality we use the triangle inequality, for the second we use~\eqref{eqn:rauch2} and~\eqref{eqn:rauch1}. Lastly, we apply the stability of the subdivision scheme. Iterating the latter yields the required bound.
	\end{proof}
	We present two brief comments on Theorem~\ref{ManifoldStabilityTHM}. First, bounding the sectional curvature from below with a positive number clearly does not change the conclusion. Still, if the lower bound is negative, such as in hyperbolic manifolds, estimations~\eqref{eqn:rauch1}-\eqref{eqn:rauch2} do not hold, and more delicate argument is needed. Second, we allow the details to differ only by their mutual angle and not magnitude. We may remove this obstacle using a more technical calculation which we omit here for compactness.
	
	\begin{remark}
		Following the methodology of~\cite{grohs2012definability}, together with our estimation~\eqref{P_proposition}, an analogue of~\eqref{eqn:stability} can be achieved based on proximity to the linear counterparts of our operators. This result is more of asymptotic flavor and it carries less information about the constant $L$. Nevertheless, it holds for more general class of manifolds.
	\end{remark}
	The stability results support the concept of using the inverse transform~\eqref{ManifoldNonInterpolatoryReconstruction} for different numerical tasks, as we will see in the next section.
	
	
	\section{Numerical Examples}\label{NumericalExamples}
	
	In this section, we focus on demonstrating our pyramid transform numerically. We begin with an illustration of the bounds from Theorem~\ref{LinearTruncatedDecayTHM} and Theorem~\ref{LinearNormalizedDecayTHM}, emphasizing the importance of mask normalization. Then, we show the application of our multiscale transforms over manifold data to the tasks of denoising and anomaly detection. All MATLAB scripts that include the examples of this section are available online at \href{https://github.com/WaelMattar/Manifold-Multiscale-Representations}{https://github.com/WaelMattar/Manifold-Multiscale-Representations} for reproducibility.
	
	\subsection{Comparing the novel linear decimation operators}\label{LinearNumericalExamples}
	
	In Section~\ref{ApproximatedDecimationSection}, we present new methods for truncating the sequence $\boldsymbol{\gamma}$ to obtain a finite mask. Comparing Theorem~\ref{LinearTruncatedDecayTHM} and Theorem~\ref{LinearNormalizedDecayTHM}, and in particular, their upper bounds on the norms of the generated detail coefficients, shows a significant additional factor in~\eqref{TruncatedGammaBound}. This section examines the numerical nature of this difference and how accurate the description of the detail coefficients' decay according to the theoretical bound is.
	
	Our example is conducted in the functional setting, where we consider the samples of the smooth periodic function $f(x)=\sin(3x)$. We choose $\mathcal{S}_{\boldsymbol{\alpha}}$ to be the linear cubic subdivision scheme, as appears in Example~\ref{CubicExample} and sample $f$ over the interval $[0,2\pi]$ at $10\times 2^{10}$ equispaced points, that is, to obtain $\boldsymbol{c}^{(J)}$ with $J=10$. The samples are treated as a periodic sequence, so it represents a bi-infinite sequence. Then, we decompose the samples via the linear multiscale transforms~\eqref{LinearTruncatedMultiscaleTransform} and~\eqref{LinearNormalizedMultiscaleTransform} which depend on the truncated mask~\eqref{Truncation} and shift invariant mask~\eqref{Normalization}, respectively. 
	
	The maximum norms of the generated details are depicted in Figure~\ref{truncation_error} as a function of the level $\ell=1,\dots,10$, for two different truncation parameters $\varepsilon=10^{-2}$ and $\varepsilon=10^{-5}$. The results show behavior that agrees with the upper bounds of Theorem~\ref{LinearTruncatedDecayTHM} and Theorem~\ref{LinearNormalizedDecayTHM}. In particular, the details, as generated by~\eqref{LinearTruncatedMultiscaleTransform}, are bounded by a value of order $\varepsilon$, due to the additional term in~\eqref{TruncatedGammaBound} which does not decay with respect to $\ell$. On the other hand, the detail coefficients generated by~\eqref{LinearNormalizedMultiscaleTransform} decay geometrically, as expected, see Corollary~\eqref{LinearCorollary}.
	
	\begin{figure}[H]
		\centering
		\begin{subfigure}[b]{0.49\textwidth}
			\includegraphics[width=1\textwidth]{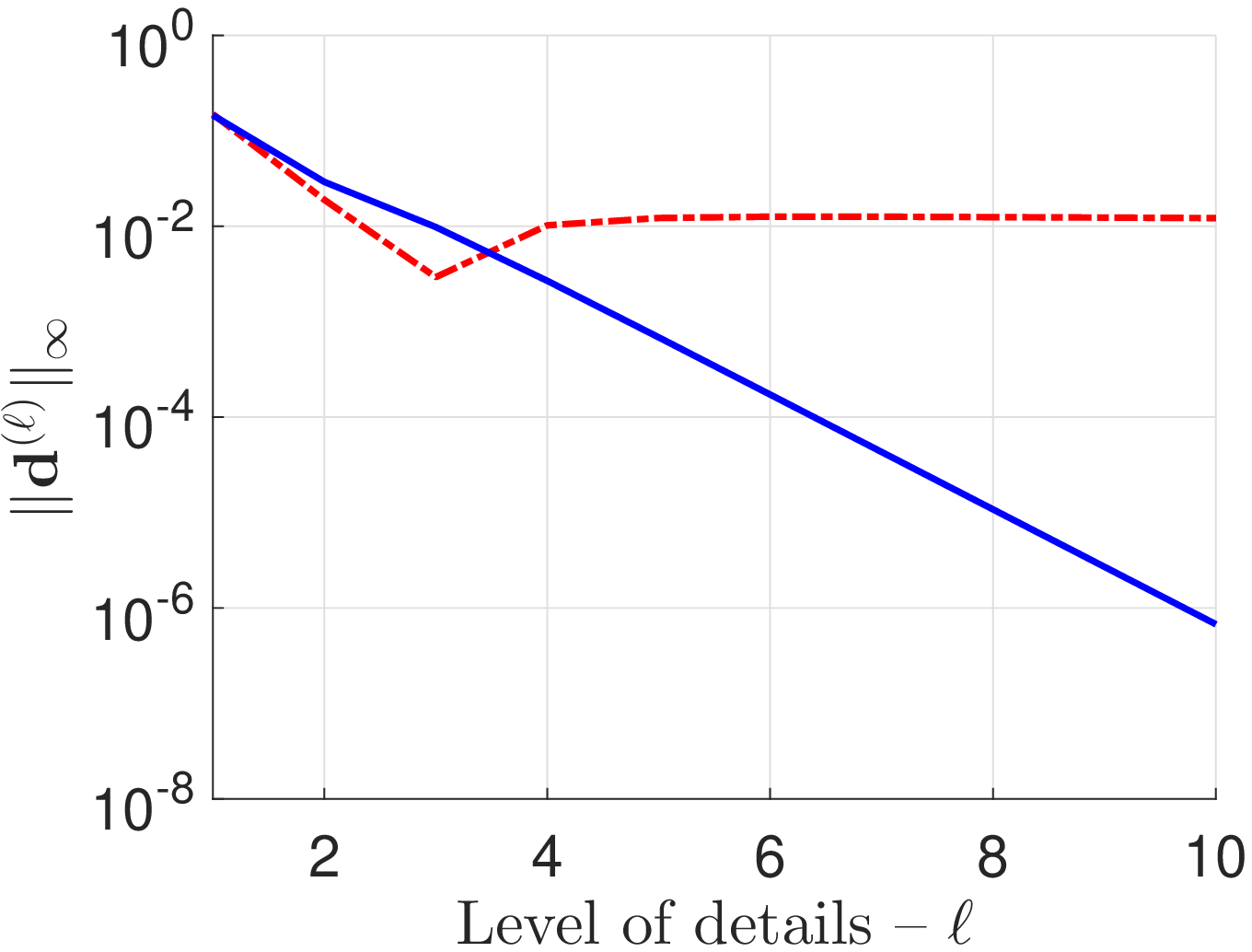}
		\end{subfigure}%
		\begin{subfigure}[b]{0.49\textwidth}
			\includegraphics[width=1\textwidth]{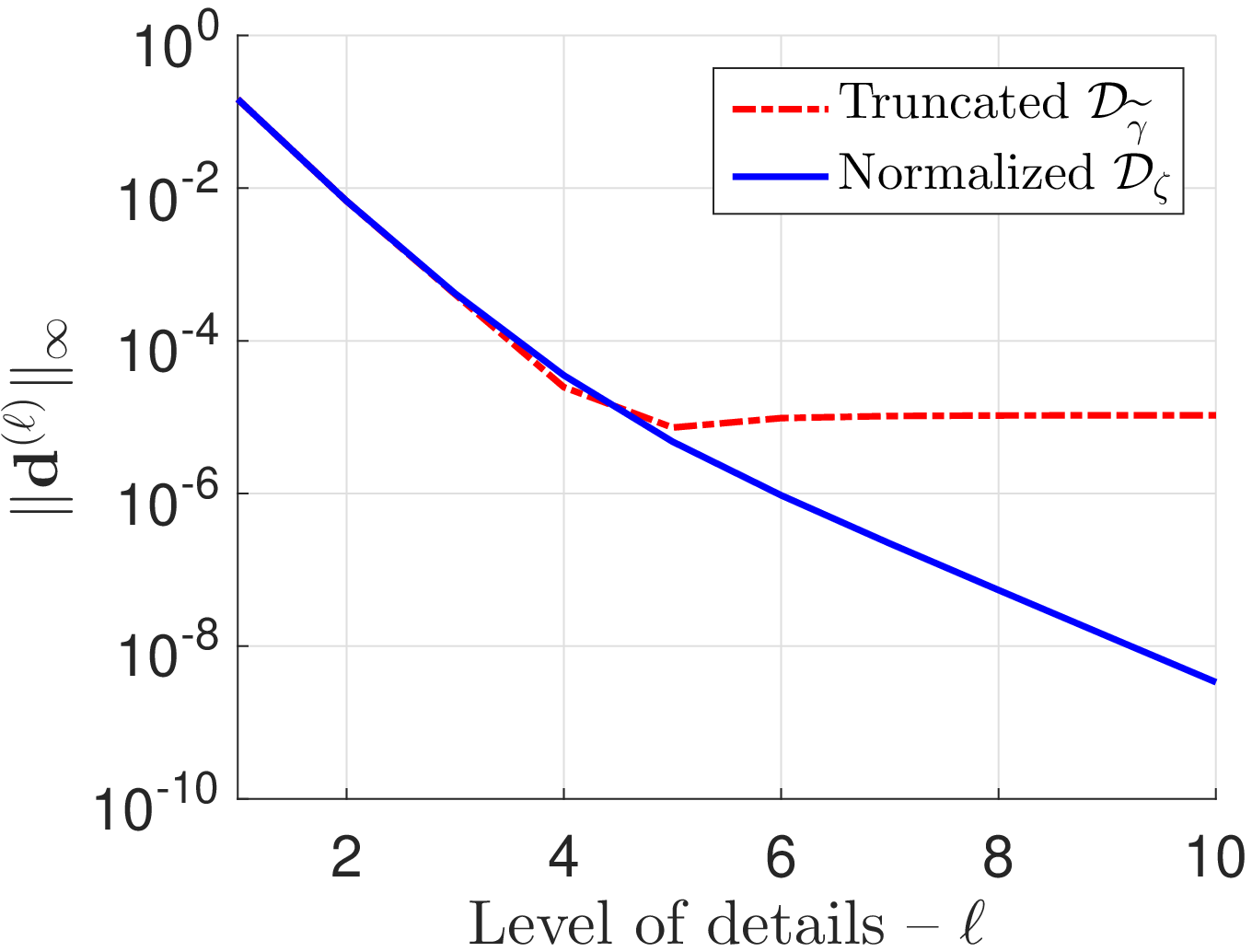}
		\end{subfigure}
		
		\caption{Truncating the decimation operator: the detail coefficients norm as a function of the level $\ell$, plotted on a logarithmic scale for different truncation parameters. On the left, $\varepsilon=10^{-2}$. On the right, $\varepsilon=10^{-5}$. In both figures, the red dashed lines correspond to $\mathcal{D}_{\widetilde{\boldsymbol{\gamma}}}$ of the truncated mask~\eqref{Truncation}, and the blue lines correspond to $\mathcal{D}_{\boldsymbol{\zeta}}$ of the truncated, shift invariant mask~\eqref{Normalization}. Note how the blue graphs are bounded by geometrically decaying bounds, as guaranteed by Corollary~\ref{LinearCorollary}. In contrast, the red dashed lines are bounded below by values of the same order as the truncation parameter $\varepsilon$, as implied by Corollary~\ref{LinearCorollary}.}
		\label{truncation_error}
	\end{figure}
	
	\subsection{Denoising of sphere-valued curve}\label{Denoising_Section}
	
	We turn to manifold-valued data and consider the unit sphere $\mathbb{S}^2$ in $\mathbb{R}^3$ as the manifold of this section. The following example serves as a proof of concept for the application of pyramid transform for curves over manifolds. Specifically, we address the problem of estimating a curve from its noisy samples. To this purpose, we follow the conventional algorithm of reconstructing the object from its thresholded multiscale coefficients. For the data model, denote by $\Gamma_k$, $k\in\mathbb{Z}$ the equidistant samples of a curve $\Gamma$ over the sphere, and by
	\begin{align}~\label{noise}
		\Upsilon_k= \Gamma_k \oplus \chi_k,
	\end{align}
	the noisy samples, where $\chi_k\sim \mathcal{N}(\boldsymbol{\mu},\boldsymbol{\Sigma})$ are i.i.d. normally distributed random variables with zero mean and covariance matrix $\boldsymbol{\Sigma}=\sigma^2 I$. The noise terms $\chi_k$ are in the respective tangent spaces $T_{\Gamma_k}\mathcal{M}$, which are isomorphic to $\mathbb{R}^2$. Note that small noise levels guarantee $\chi_k$ to be within the injectivity radius of the exponential map $\oplus$ associated to point $\Gamma_k$. We, therefore, assume that the realizations of the noise terms are sufficiently small.
	
	In the current test case, we take $\Gamma$ to be a flower-like periodic smooth $\mathbb{S}^2$-valued curve defined via spherical coordinates as,
	\begin{align}~\label{sphere_flower_curve}
		\Gamma(\theta) = \big(\sin(\varphi(\theta))\cos(\theta),\; \sin(\varphi(\theta))\sin(\theta),\; \cos(\varphi(\theta))\big),\quad \varphi(\theta)=\displaystyle\frac{\pi}{16}\cos(N\theta)+\frac{\pi}{6}, \quad \theta\in[0,2\pi].
	\end{align}
	Here, $N\in\mathbb{N}$ determines the number of the flower's leaves. We set $N=5$ as shown in Figure~\ref{sphere_samples}.
	
	Let $\mathcal{T}_{\boldsymbol{\alpha}}$ be the Riemannian analogue of the cubic spline subdivision scheme adapted to $\mathbb{S}^2$ as described in Section~\ref{RiemannianAnalogue}. Denote by $\mathcal{Y}_{\boldsymbol{\zeta}}$ its approximated decimation operator with the shift invariant mask $\boldsymbol{\zeta}$, as given in~\eqref{ManifoldNonInterpolatingTransform}. In this example, we pick $\varepsilon=10^{-5}$ which induces that $\boldsymbol{\zeta}$ consists of $13$ nonzero elements. We note that $\mathbb{S}^2$ is a 2-dimensional topological manifold with positive sectional curvature, thus, optimization problems like~\eqref{ManifoldFrechetMean} and~\eqref{ManifoldFrechetMeanDecimation} may have infinite solutions, e.g., when averaging two antipodal points. However, for close enough points on $\mathbb{S}^2$, the center of mass exists uniquely, see~\cite{dyer2016barycentric, SvenjaWallnerConvergenceSphere}. We follow a Riemannian gradient descent method~\cite{krakowski2007computation} to calculate the Riemannian center of mass on $\mathbb{S}^2$. Figure~\ref{sphere_decomposition} demonstrates $\Gamma$ of~\eqref{sphere_flower_curve} alongside its corresponding pyramidical representation via our multiscale transform~\eqref{ManifoldNonInterpolatingTransform}, which manifests the detail coefficients decay.
	
	\begin{figure}
		\centering
		\begin{subfigure}[b]{0.43\textwidth}
			\includegraphics[width=0.9\textwidth]{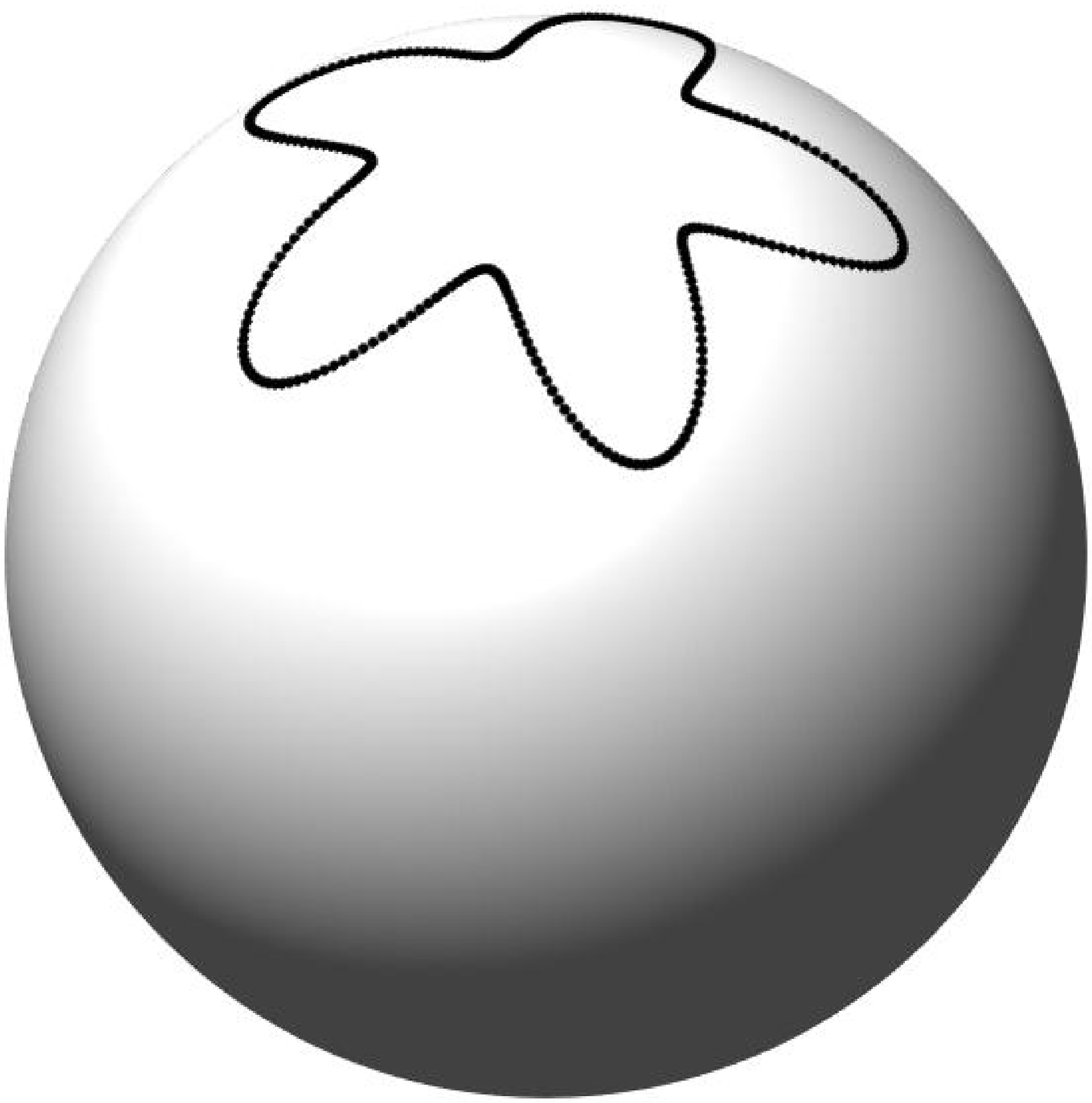}
			\caption{}
			\label{sphere_samples}
		\end{subfigure}
		\hfill
		\begin{subfigure}[b]{0.53\textwidth}	
			\centering
			\includegraphics[width=1\textwidth]{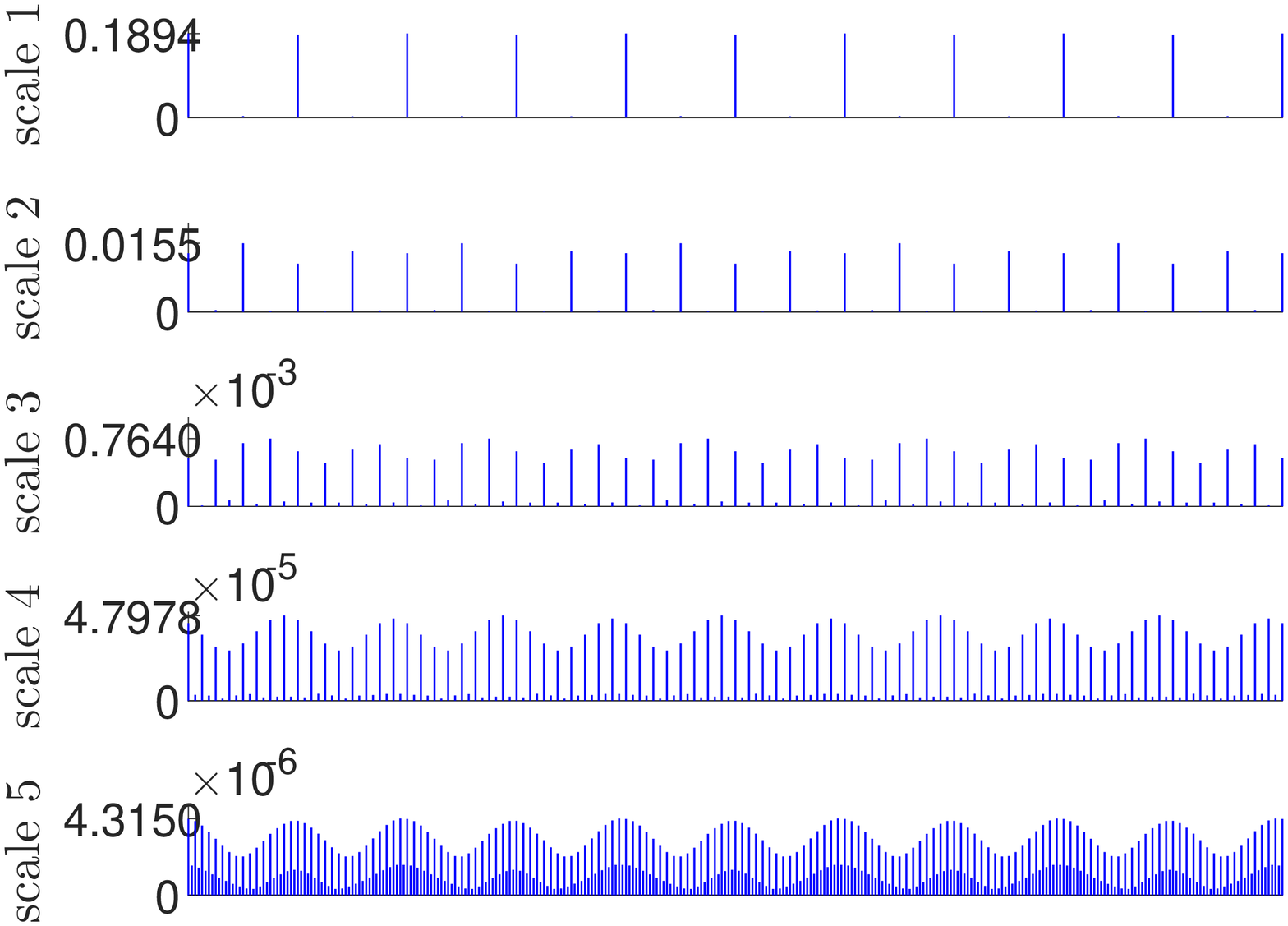}
			\caption{}
			\label{sphere_multiscale}
		\end{subfigure}
		\caption{The curve $\Gamma$ of~\eqref{sphere_flower_curve} over the unit sphere with its multiscale transform~\eqref{ManifoldNonInterpolatingTransform}. On the left, $10\times 2^5$ equispaced samples of the flower-like curve. On the right, the Euclidean norms of the detail coefficients $d^{(\ell)}_k$ for $\ell=1,\dots,5$, generated by applying~\eqref{ManifoldNonInterpolatingTransform}. As the scale increases, the maximal norm of each layer decays geometrically as guaranteed by Corollary~\ref{ManifoldCorollary}. Note that every second element of each layer is smaller. This phenomenon is explained by condition~\eqref{NonInterpolatingCondition} and reflects the approximation character of our decimation operator. Therefore, smaller values of the truncation parameter $\varepsilon$ yield to smaller details coefficients.}
		\label{sphere_decomposition}
	\end{figure}
	
	We now synthetically generate noisy samples according to the model~\eqref{noise}, with $\sigma \approx 1/80$. Figure~\ref{sphere_noise} shows the noisy data alongside their corresponding pyramidical representation via our multiscale transform~\eqref{ManifoldNonInterpolatingTransform}. As we can see, the multiscale representation of the noisy sequence $\Upsilon$ does not enjoy the property of detail coefficients decay.
	
	To estimate $\Gamma$ from its noisy samples $\Upsilon_k$, we follow~\cite{donoho1995noising}, where it is shown that thresholding of the details of the pyramid transform yields a nearly-optimal estimation. In other words, we go over each layer of multiscale coefficients corresponding to the noisy curve, see Figure~\ref{sphere_noisy_details}, and set to zero all detail coefficients with norm below a fixed threshold, $0.14$ in our case. This process yields to a sparser pyramid representation which forms an estimation of the ground truth $\Gamma$. The approximant is synthesized iteratively by~\eqref{ManifoldNonInterpolatoryReconstruction}. 
	
	Figure~\ref{sphere_denoised} demonstrates the denoised curve alongside its multiscale representation. Indeed, the detail coefficients of the denoised curve are bounded by a geometrically decreasing sequence, which indicates the smoothness of the resulted curve.
	
	To sum, our multiscale transform~\eqref{ManifoldNonInterpolatingTransform} makes a useful tool for denoising curves over manifolds. The denoising's performance in this example is reflected by the resemblance between the ground truth and the denoised curves. 
	
	\begin{figure}
		\centering
		\begin{subfigure}[b]{0.43\textwidth}
			\centering
			\includegraphics[width=0.9\textwidth]{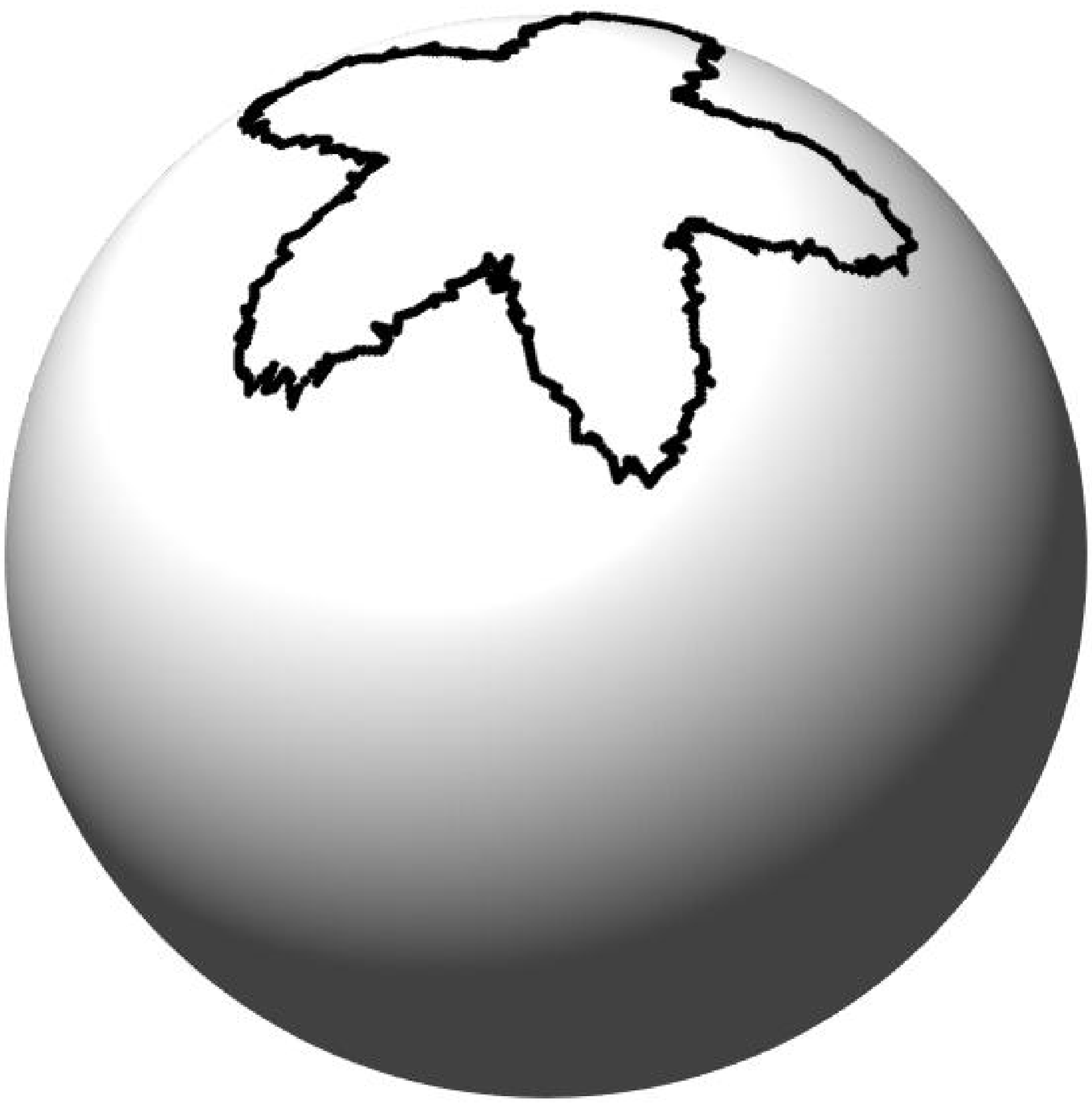}
			\caption{}
			\label{sphere_noisy_points}
		\end{subfigure}
		\hfill
		\begin{subfigure}[b]{0.53\textwidth}
			\centering
			\includegraphics[width=1\textwidth]{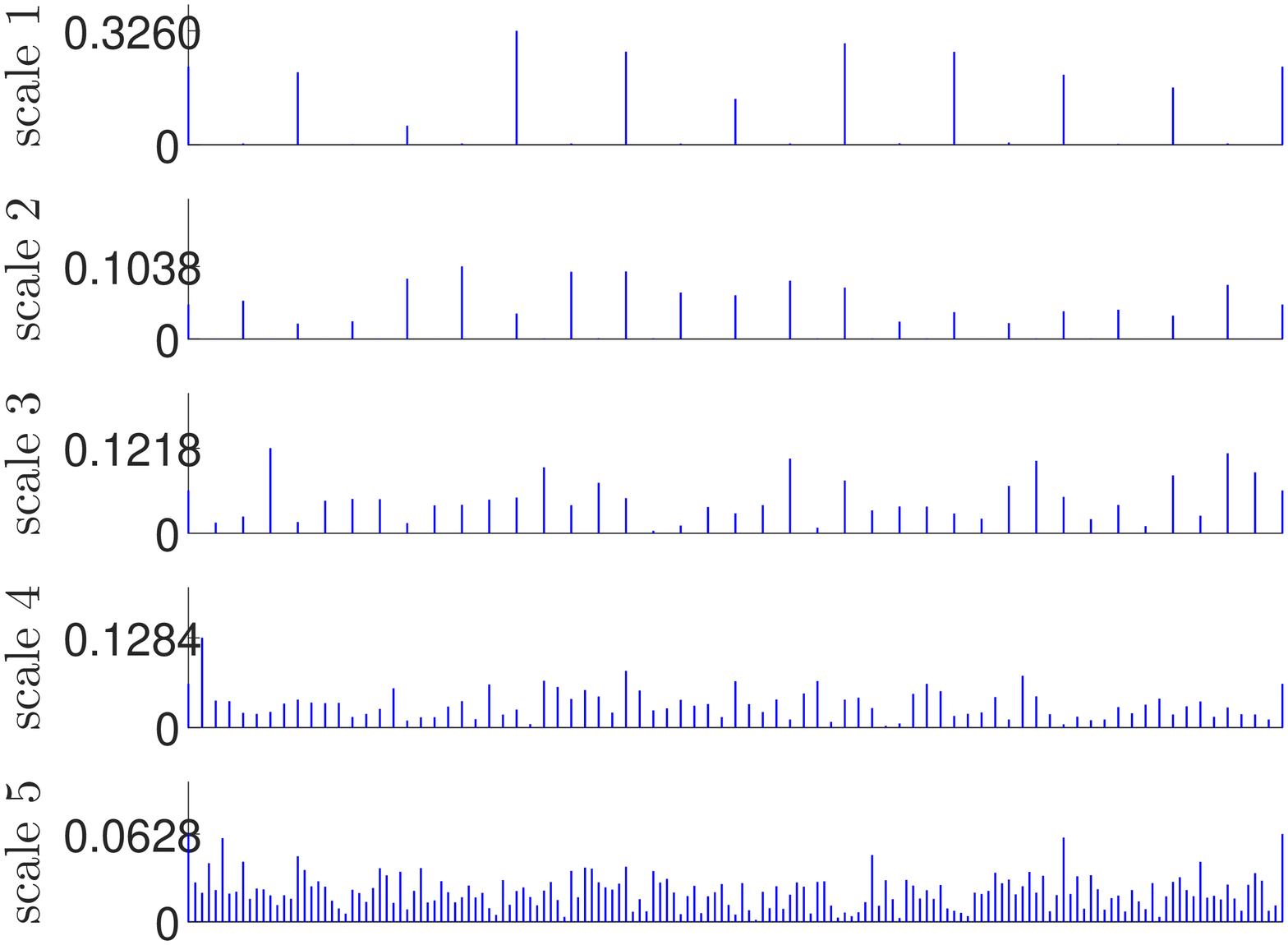}
			\caption{}
			\label{sphere_noisy_details}
		\end{subfigure}
		\caption{Noisy samples and its multiscale representation. On the left, the noisy points $\Upsilon$ of~\eqref{noise}. On the right, the Euclidean norms of the detail coefficients.}
		\label{sphere_noise}
	\end{figure}
	
	\begin{figure}
		\centering
		\begin{subfigure}[b]{0.43\textwidth}
			\centering
			\includegraphics[width=0.9\textwidth]{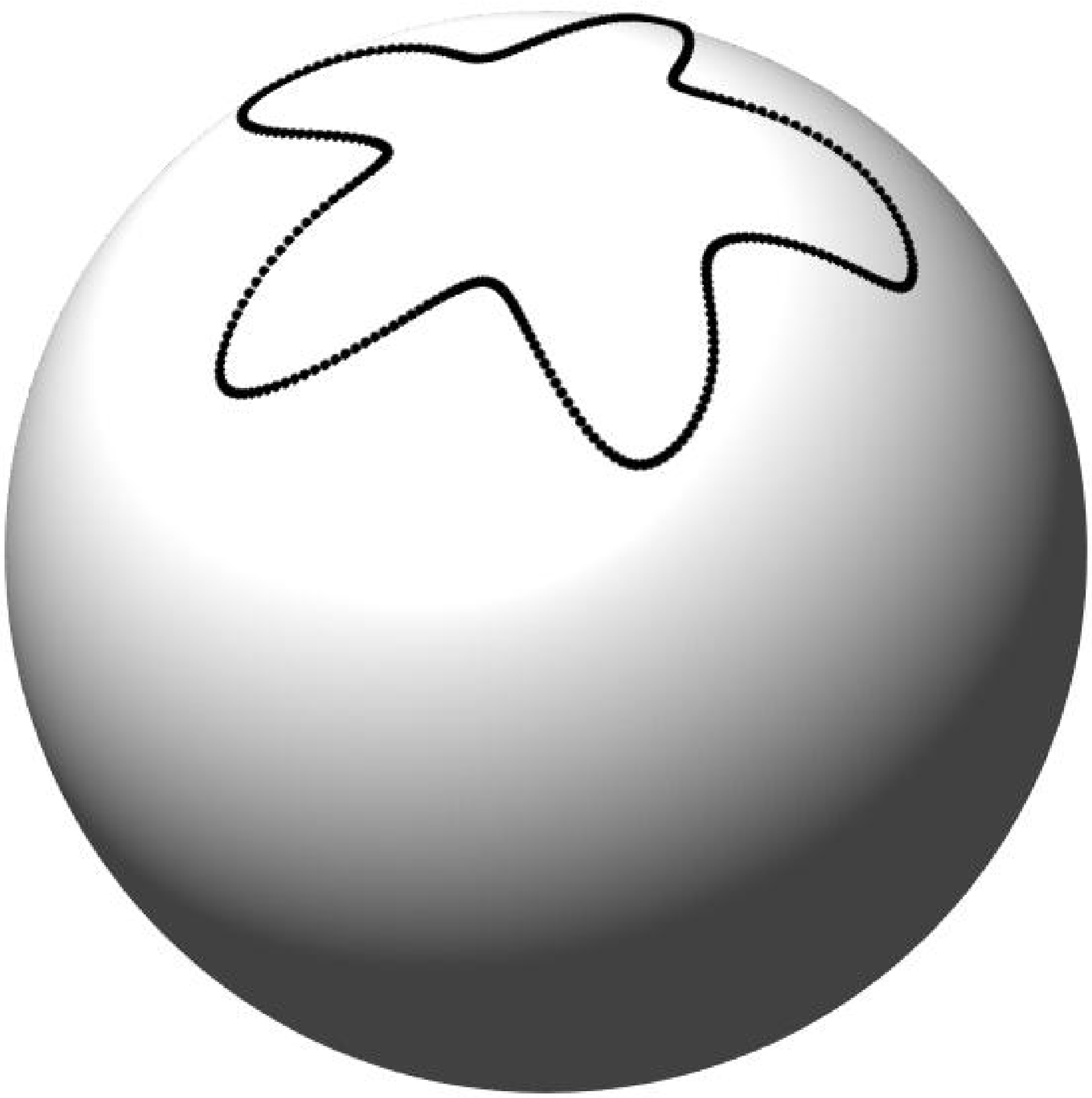}
			\caption{}
			\label{sphere_denoised_points}
		\end{subfigure}
		\hfill
		\begin{subfigure}[b]{0.53\textwidth}
			\centering
			\includegraphics[width=1\textwidth]{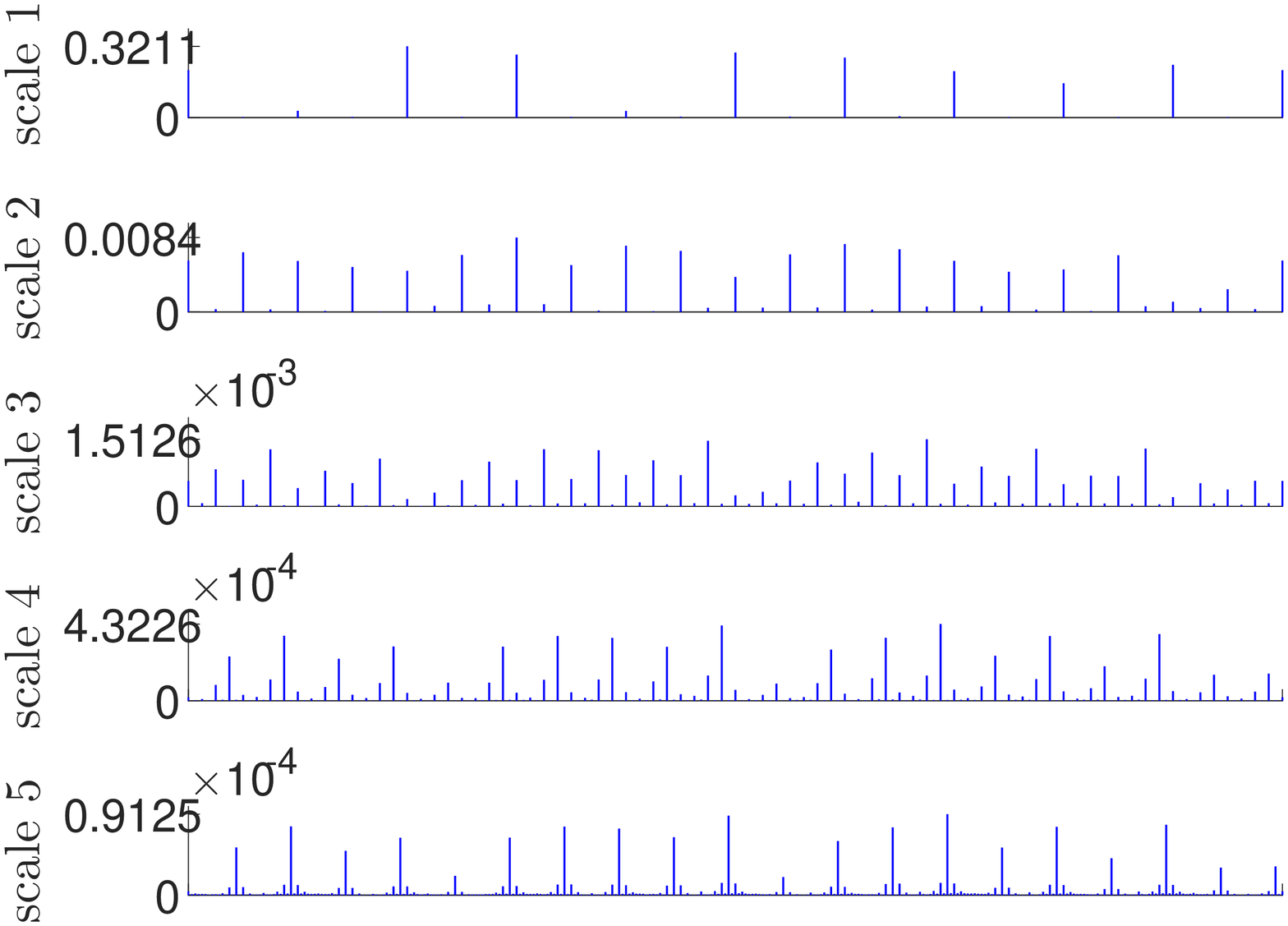}
			\caption{}
			\label{sphere_denoised_details}
		\end{subfigure}
		\caption{Denoised curve and its multiscale representation. On the left, estimation of $\Gamma$~\eqref{sphere_flower_curve}. On the right, the Euclidean norms of the detail coefficients. The decay of detail coefficients indicate the smoothness of the denoised curve.}
		\label{sphere_denoised}
	\end{figure}
	
	\subsection{Anomaly detection of \texorpdfstring{$\mathcal{SPD}(3)$}{SPD(3)}-valued curve}\label{Anomaly_Detection_Section}
	
	Our multiscale transform~\eqref{ManifoldNonInterpolatingTransform} involves the application of two local operators. This feature makes the transform a beneficial tool for detecting and analyzing local behavior in manifold-valued curves. This section focuses on representing curves over the cone of $3 \times 3$ symmetric positive matrices, which we denote by $\mathcal{SPD}(3)$. In particular, we show the application of our pyramid analysis to the problem of anomaly detection. Namely, we aim to automatically detect rapid local changes in a time series of matrices by inspecting its multiscale representation. 
	
	We consider a smooth periodic $\mathcal{SPD}(3)$-valued curve given explicitly via trigonometric deformations. Then, we apply a scaling factor to the eigenvalues of all the matrices that fall in the middle third of the curve to provide anomaly. This application gives rise to a piecewise smooth $\mathcal{SPD}(3)$-valued curve with two jump discontinuities. We depict the two curves, both the smooth original one and the distributed piecewise smooth, in Figure~\ref{spd_curve}. Each curve is represented by a series of centered ellipsoids, where every ellipsoid has its main axes determined by the eigenvectors of the corresponding  matrix and their lengths by the associated eigenvalues.
	
	We set the test by taking $\mathcal{T}_{\boldsymbol{\alpha}}$ to be the corner-cutting (quadratic B-spline) subdivision scheme, as presented in Example~\ref{QuadraticExample}, adapted to $\mathcal{SPD}(3)$ as described in Section~\ref{RiemannianAnalogue}. Denote by $\mathcal{Y}_{\boldsymbol{\zeta}}$ its approximated decimation operator with the truncation parameter $\varepsilon=10^{-4}$, implying a shift invariant mask $\boldsymbol{\zeta}$ with $9$ nonzeros. The Riemannian center of mass over $\mathcal{SPD}(3)$ is globally unique due to the manifold's nonpositive sectional curvature. To calculate it, we follow the gradient descent method in~\cite{iannazzo2019derivative}.
	
	Next, we decompose both curves of Figure~\ref{spd_curve} by the multiscale transform~\eqref{ManifoldNonInterpolatingTransform} and investigate the norms of the detail coefficients. The norms of the detail coefficients, which lie in the linear space of all symmetric matrices of order $3$, are presented in Figure~\ref{spd_details_norms}. As it turns out, the detail coefficients corresponding to the smooth curve are represented by a geometrically decreasing sequence, as guaranteed by Corollary~\ref{ManifoldCorollary}. However, in the vicinities of the anomaly points, the detail coefficients generated by our multiscale transform~\eqref{ManifoldNonInterpolatingTransform} have relatively large norms. Namely, the large detail coefficients are correlated with the parametric locations around the jump discontinuities. Therefore, the multiscale transform~\eqref{ManifoldNonInterpolatingTransform} makes a useful tool for detecting such anomalies. 
	
	\begin{figure}
		\begin{center}
			\begin{subfigure}[b]{\linewidth}
				\centering
				\includegraphics[width=\textwidth]{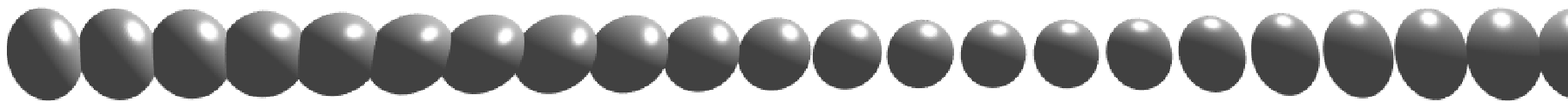}
				\caption{}
				\label{spd_samples}
			\end{subfigure}
			\bigskip
			\begin{subfigure}[b]{\linewidth}
				\centering
				\includegraphics[width=\textwidth]{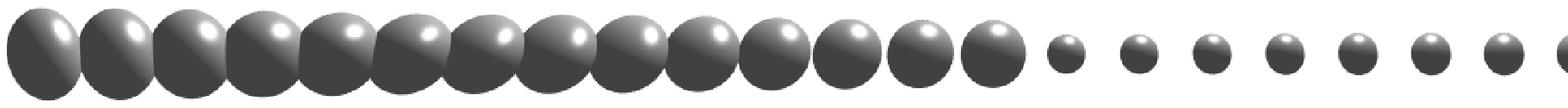}
				\caption{}
				\label{spd_heaviside_samples}
			\end{subfigure}
			\caption{$\mathcal{SPD}(3)$-valued curves. On the top, $41$ ellipsoids that represent smooth $\mathcal{SPD}(3)$-valued curve. On the bottom, the modified $41$ ellipsoids now represent the piecewise-smooth curve, with two jump discontinuities in the middle.}
			\label{spd_curve}
		\end{center}
	\end{figure}
	
	\begin{figure}
		\begin{center}
			\begin{subfigure}[b]{0.48\linewidth}
				\centering
				\includegraphics[width=\linewidth]{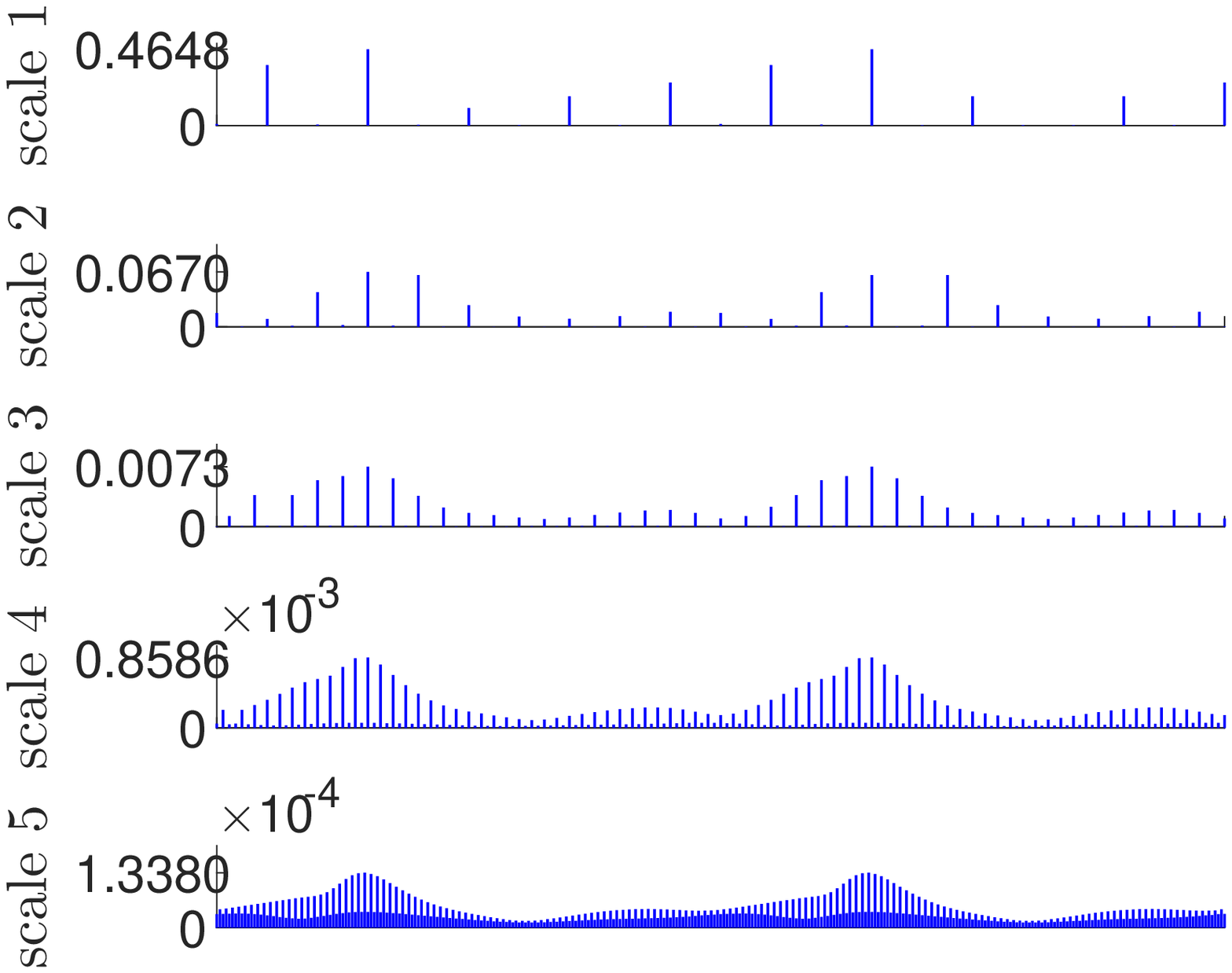}
				\caption{}
				\label{spd_multiscale_details}
			\end{subfigure}
			\begin{subfigure}[b]{0.48\linewidth}
				\centering
				\includegraphics[width=\linewidth]{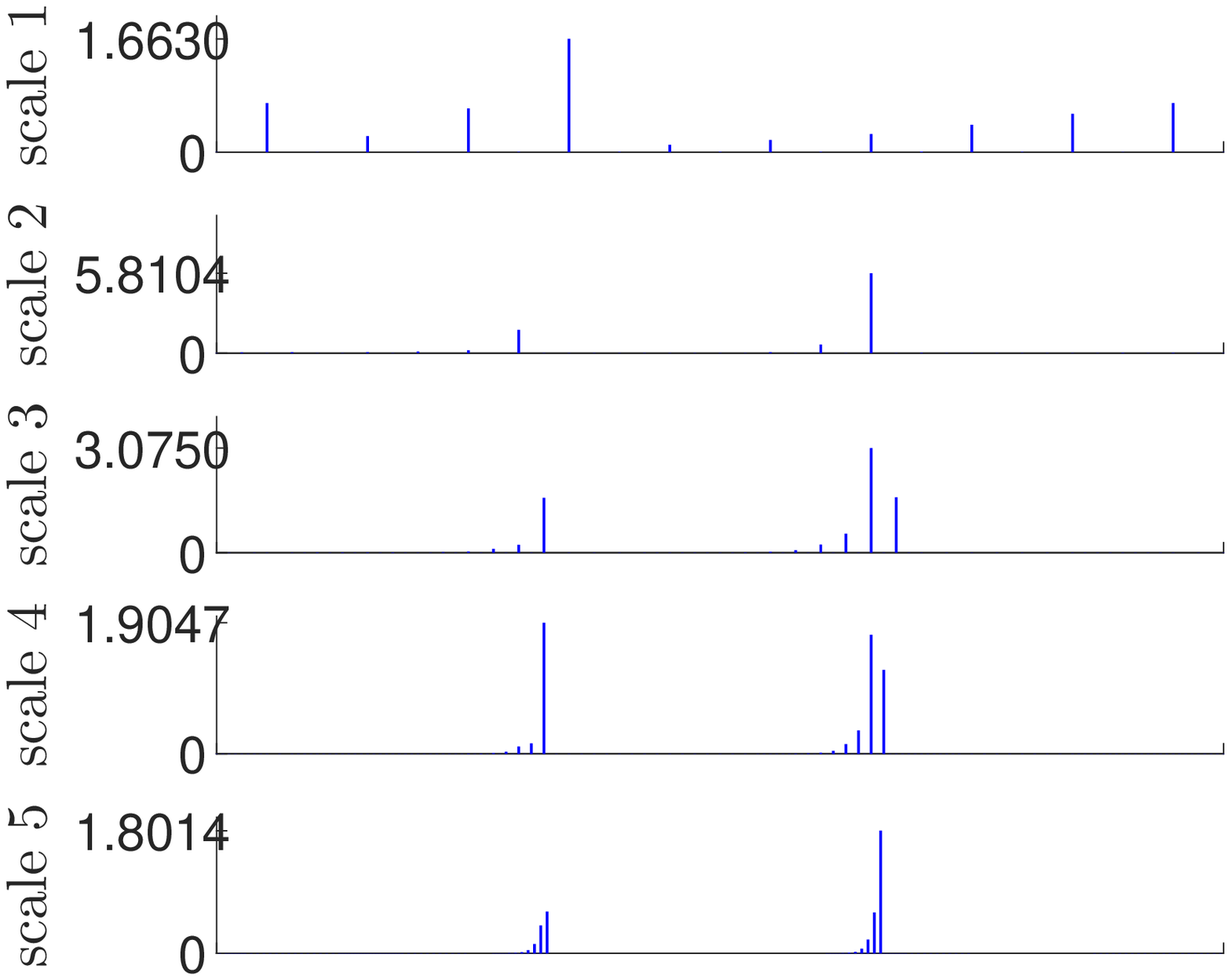}
				\caption{}
				\label{spd_heaviside_multiscale_details}
			\end{subfigure}
			\caption{Frobenius norms of the detail coefficients of the two $\mathcal{SPD}(3)$ curves in Figure~\ref{spd_curve}. On the left, detail coefficients norms corresponding to Figure~\ref{spd_samples}. On the right, detail coefficients norms corresponding to Figure~\ref{spd_heaviside_samples}. The decay rate, which decreases with each layer in~(\subref{spd_multiscale_details}), implies the curve's smoothness. Moreover, note how the theoretical condition~\eqref{NonInterpolatingCondition} is illustrated in~(\subref{spd_multiscale_details}) as every second detail is proportional to the truncation parameter $\varepsilon$. On the other hand, the two local peaks in~(\subref{spd_heaviside_multiscale_details}) indicate radical changes in the respective curve and reveal the abnormalities.}
			\label{spd_details_norms}
		\end{center}
	\end{figure}
	
	\begin{remark}~\label{Remark_P}
		We numerically estimated the constant $P$ of~\eqref{P_proposition} corresponding to this section's manifold settings. The results appear in Table~\ref{Sphere_Table} and Table~\ref{SPD_Table} at Appendix~\ref{app:constant_P} where we present the minimal possible $P$. This value decreases monotonically to $1$ as the scale of sampling, $J$, increases. This phenomenon implies that the decimation operation $\mathcal{Y}_{\boldsymbol{\zeta}}$ behaves like the simple downsampling operation for close enough $\mathcal{M}$-valued data points.
	\end{remark}
	
	\section*{Acknowledgement}
	
	The authors thank Nira Dyn and David Levin for their helpful comments and insightful discussions. N.S. was partially supported by BSF-NSF grant no. 2019752 and BSF grant no. 2018230.
	
	\bibliographystyle{plain}
	\bibliography{References}
	
	\appendix
	
	\section{Numerical evaluation of the decaying factor}
	\label{app:constant_P} 
	
	Lemma~\ref{ManifoldDecayLEM} introduces a decaying rate of the norms of the details. Here, we provide several numerical evaluations of the decaying factor $P$ of~\eqref{P_proposition}, as observed in the examples of Sections~\ref{Denoising_Section} and~\ref{Anomaly_Detection_Section}. Indeed, there exists a constant $P>1$ such that~\eqref{P_proposition} holds for sequences $\boldsymbol{c}$ sampled equidistantly from a differentiable curve over arc-length parametrization. In particular, as seen through the proof of Lemma~\ref{ManifoldDecayLEM}, the minimal possible value of $P$ can be evaluated by
	$$P_{\text{min}}=\frac{\Delta_{\mathcal{M}}(\mathcal{Y}_{\boldsymbol{\zeta}}\boldsymbol{c})}{2\Delta_{\mathcal{M}}(\boldsymbol{c})},$$
	where $\mathcal{Y}_{\boldsymbol{\zeta}}$ is the decimation operator used in the multiscale transform~\ref{ManifoldNonInterpolatingTransform}. Under the settings of Sections~\ref{Denoising_Section} and~\ref{Anomaly_Detection_Section}, we calculate $P_{\text{min}}$ for different $\Delta_{\mathcal{M}}(\boldsymbol{c})$ values. The results are shown in Table~\ref{Sphere_Table} and Table~\ref{SPD_Table}.
	
	The main feature of Table~\ref{Sphere_Table} and Table~\ref{SPD_Table} is that, in both manifold settings, the constant $P_{\text{min}}$ decreases monotonically to $1$ as $\Delta_{\mathcal{M}}(\boldsymbol{c})$ decreases. This fact indicates a similar behavior between $\mathcal{Y}_{\boldsymbol{\zeta}}$ and the downsampling operation $\downarrow 2$, when the distance between data points reduces, as stated in Remark~\ref{Remark_P}.
	
	\begin{table}[h!]
		\centering
		\begin{tabular}{| c || c | c | c | c | c | c | c | c ||}
			\hline
			$\Delta_{\mathcal{M}}(\boldsymbol{c})$ & 0.2667 & 0.1639 & 0.0859 & 0.0433 & 0.0217 & 0.0108 & 0.0054 & 0.0027 \\ [0.5ex] 
			\hline\hline
			$\mathbb{S}^2$ & 1.4021 & 1.0368 & 1.0205 & 1.0086 & 1.0038 & 1.0003 & 1.0001 & 1.0000 \\ [0.5ex] 
			\hline
		\end{tabular}
		\caption{The constant $P_{\text{min}}$ against different values of $\Delta_{\mathcal{M}}(\boldsymbol{c})$. The second row demonstrates $P_{\text{min}}$ corresponding to samples of the curve $\Gamma$ in~\eqref{sphere_flower_curve} over the sphere $\mathbb{S}^2$, with the respective $\Delta_{\mathcal{M}}(\boldsymbol{c})$ value. The operator $\mathcal{Y}_{\boldsymbol{\zeta}}$ is the adapted even-inverse of the cubic B-spline with truncation parameter $\varepsilon=10^{-5}$, as in Section~\ref{Denoising_Section}.}
		\label{Sphere_Table}
	\end{table}
	
	\begin{table}[h!]
		\centering
		\begin{tabular}{| c || c | c | c | c | c | c | c | c ||}
			\hline
			$\Delta_{\mathcal{M}}(\boldsymbol{c})$ & 0.6837 & 0.3542 & 0.1813 & 0.0912 & 0.0457 & 0.0228 & 0.0114 & 0.0057 \\ [0.5ex] 
			\hline\hline
			$\mathcal{SPD}(3)$ & 1.2661 & 1.0613 & 1.0176 & 1.0053 & 1.0014 & 1.0004 & 1.0000 & 1.0000 \\ [0.5ex] 
			\hline
		\end{tabular}
		\caption{The constant $P_{\text{min}}$ against different values of $\Delta_{\mathcal{M}}(\boldsymbol{c})$. The second row demonstrates $P_{\text{min}}$ corresponding to samples of the curve shown in Figure~\ref{spd_samples} over the manifold $\mathcal{SPD}(3)$, with the respective $\Delta_{\mathcal{M}}(\boldsymbol{c})$ value. The operator $\mathcal{Y}_{\boldsymbol{\zeta}}$ is the adapted even-inverse of the quadratic B-spline with truncation parameter $\varepsilon=10^{-4}$, as in Section~\ref{Anomaly_Detection_Section}.}
		\label{SPD_Table}
	\end{table}
	
\end{document}